\documentclass[oneside]{amsart}
\usepackage{amssymb}
\usepackage{amsmath}
\usepackage{amsfonts}
\usepackage{amsthm}
\usepackage{hyperref}
\usepackage{amscd}
\usepackage{color}
\usepackage{enumerate}
\usepackage{tikz-cd}
\usepackage{ytableau}

\newcommand{\hl}[1]{{\mathbf{\textcolor{blue}{#1}}}}
\newcommand{\Rf}{\mathfrak{R}}
\newcommand{\ram}{\mathfrak{r}}
\newcommand{\ZZ}{\mathbb{Z}}
\newcommand{\KK}{\mathbb{K}}

\newcommand{\PP}{\mathbb{P}}

\newcommand{\NN}{\mathbb{N}}

\newcommand{\CO}{\mathcal{O}}
\newcommand{\val}{\mathfrak{v}}

\newcommand{\mcL}{\mathcal{L}}

\newcommand{\F}{\mathcal{F}}

\newcommand{\mfm}{\mathfrak{m}}
\newcommand{\mfS}{\mathfrak{S}}

\DeclareMathOperator{\codim}{codim}

\DeclareMathOperator{\proj}{Proj}

\DeclareMathOperator{\ord}{ord}

\newtheorem{thm}{Theorem}[section]
\newtheorem{prop}[thm]{Proposition}
\newtheorem{lemma}[thm]{Lemma}
\newtheorem{cor}[thm]{Corollary}

\theoremstyle{definition}

\newtheorem{defn}[thm]{Definition}
\newtheorem{rem}[thm]{Remark}
\newtheorem{ex}[thm]{Example}
\newtheorem{con}[thm]{Construction}

\tikzset{
  treenode/.style = {shape=rectangle, rounded corners,
                     draw, align=center,
                     top color=white, bottom color=blue!20},
  root/.style     = {treenode, font=\Large, bottom color=red!30},
  env/.style      = {treenode, font=\ttfamily\normalsize},
  dummy/.style    = {circle,draw}
}
\ytableausetup{boxsize=.5em}
\newcommand{\tb}{
$\overbrace{\begin{ytableau}
\ &\none &\none[\cdots]&\none &\ &\ &\\
\end{ytableau}}^n
$}
\newcommand{\tbc}{
$\overbrace{\begin{ytableau}
\ &\none &\none[\cdots]&\none &\ &\ &\\
\ &\none &\none[\cdots]&\none &\ &\ &\\
\end{ytableau}}^n
$}
\newcommand{\tbcd}{
$\overbrace{\begin{ytableau}
\ &\none &\none[\cdots]&\none &\ &\ &\\
\ &\none &\none[\cdots]&\none &\ &\ &\\
\ &\none &\none[\cdots]&\none &\ &\ &\\
\end{ytableau}}^n
$}
\newcommand{\tbd}{
$\overbrace{\begin{ytableau}
\ &\none &\none[\cdots]&\none &\ &\ &\\
\ &\none &\none[\cdots]&\none &\ &\ \\
\end{ytableau}}^n
$}
\newcommand{\tbce}{
$\overbrace{\begin{ytableau}
\ &\none &\none[\cdots]&\none &\ &\ &\\
\ &\none &\none[\cdots]&\none &\ &\ &\\
\ &\none &\none[\cdots]&\none &\ &\ \\
\end{ytableau}}^n
$}
\newcommand{\tbcf}{
$\overbrace{\begin{ytableau}
\ &\none &\none[\cdots]&\none &\ &\ &\\
\ &\none &\none[\cdots]&\none &\ &\ &\\
\ &\none &\none[\cdots]&\none &\  \\
\end{ytableau}}^n
$}

\usepackage[textsize=tiny]{todonotes}

\title[Singular Curves of Low Degree]{Singular Curves of Low Degree and Multifiltrations from Osculating Spaces}

\author{Jaros{\l}aw Buczy{\'n}ski}
\address{Institute of Mathematics of the Polish Academy of Sciences, ul.~{\'S}niadeckich 8, \mbox{00-656}~Warsaw, Poland}
\email{jabu@mimuw.edu.pl}

\author{Nathan Ilten}
\address{Department of Mathematics, Simon Fraser University,
8888 University Drive, Burnaby BC V5A1S6, Canada}
\email{nilten@sfu.ca}

\author{Emanuele Ventura}
\address{Department of Mathematics, Texas A\&M University, College Station, TX~\mbox{77843-3368}, USA}
\email{eventura@math.tamu.edu, emanueleventura.sw@gmail.com}

\begin{document}
\begin{abstract}
In order to study projections of smooth curves, we introduce multifiltrations obtained by combining flags of osculating spaces.
We classify all configurations of singularities occurring for a projection of a smooth curve embedded by a complete linear system away from a projective linear space of dimension at most two. In particular, we determine all configurations of singularities of non-degenerate degree $d$ rational curves in $\PP^n$ when $d-n\leq 3$ and $d<2n$. Along the way, we describe the Schubert cycles giving rise to these projections. 

We also reprove a special case of the Castelnuovo bound using these multifiltrations: under the assumption $d<2n$, the arithmetic genus of any non-degenerate degree $d$ curve in $\PP^n$ is at most $d-n$.
\end{abstract}
\keywords{Projective curves, curve singularities, gap function, multifiltrations, osculating space, Schubert variety}
\subjclass[2010]{Primary: 14H20; Secondary: 14N15, 14H50, 14C20}
\maketitle

\section{Introduction}
\subsection{Motivation and Background}
How bad can the singularities of a curve $C$ of degree $d$ in $\PP^n$ be? The study of this question 
is classical. A first measure of this is given by the difference between the arithmetic genus $\rho_a$ and the geometric genus $\rho_g$ of $C$. This difference is zero if and only if the curve is smooth. Moreover, fixing the numerical invariants $d$, $n$, $\rho_a$, and $\rho_g$, one could hope to classify all the configurations of singularities that can occur.

Perhaps the simplest situation is that of plane curves. Here, the arithmetic genus of any degree $d$ plane curve $C$ is $\rho_a(C)=(d-1)(d-2)/2$. 
For degrees $d=3,4,5$, all possible configurations of singularities are classically known; see e.g.~\cite{namba}. To our knowledge, a classification for $d=6$ remains an open question, although there have been partial results in this direction \cite{conner,yang,saleem}.
Furthermore, complex planar rational curves which are homeomorphic to a two dimensional sphere (that is, they only admit singularities of cuspidal type, and no multibranched singularities) are recently shown to have at most four singular points \cite{koras_palka_at_most_4_cusps} and they are partially classified.

In higher dimensions, less is known. A common approach is to view a curve $C\subset \PP^n$ as the image of a smooth curve $X\subset \PP^m$ under a linear projection $\PP^m\dashrightarrow \PP^n$; this was used extensively by Veronese \cite{veronese}. 
More generally, any curve $C$ may be viewed as the image of a smooth curve $X$ under the map obtained from some linear series of a line bundle $\mcL$.
Important information concerning the result of such a map may be obtained by studying the intersection behaviour of the orthogonal complement of this linear series with the flags of osculating planes for the curve $X$ with respect to $\mcL$. This has been used for instance by Piene \cite{piene} to recover the generalized Pl\"ucker formulas, and by Eisenbud and Harris \cite{eh1,eh2} to study linear series on curves. 

A special case is when the curve $C$ under consideration is \emph{rational}.
Any non-degenerate rational degree $d$ curve in $\PP^n$ may be obtained from the rational normal curve $X_d$ of degree $d$ by a projection $\PP^{d}\dashrightarrow \PP^n$. Thus, one may study singularities of rational curves via projections of the rational normal curve. 

\subsection{Our Approach and Results}
Throughout the paper, we work over an algebraically closed field $\KK$. A {\it curve} is a projective one-dimensional integral scheme (irreducible and reduced, but not necessarily smooth). We usually consider a curve $C$ arising as the image of a smooth projective curve $X$ under a morphism coming from a linear series of a line bundle $\mcL$.

As noted above, the information encoded in the intersection behaviour of the orthogonal complement of the linear series with the osculating flags of $X$ is useful in understanding $C$. More precisely, it gives valuable information on the singularities of individual branches of $C$. However, it completely misses the way that these branches interact.

Our strategy is to look at the $\ZZ^r$-graded multifiltration induced by the osculating flags at the $r$ points of $X$ corresponding to the $r$ branches of a given singularity of $C$. This captures much more information, and in many cases it allows us to determine the type of a singularity. 

We apply these ideas to obtain two main results. The first is as follows:

\begin{thm}[See Corollary~\ref{thm:1new}]\label{thm:1}
Let $C\subset \PP^n$ be any non-degenerate degree $d$ curve with arithmetic genus $\rho_a$.
Assume that $d<2n$. Then $\rho_a\leq d-n$. 
\end{thm}

The bound on $\rho_a$ is as strong as possible: for any $d<2n$, there is a non-degenerate rational curve $C$ with arithmetic genus $d-n$ (Example~\ref{ex:1}). Furthermore, if $d\geq 2n$ the statement fails (Example~\ref{ex:2}).
After we proved Theorem~\ref{thm:1}, we 
realized that the statement is just a special case of Castelnuovo's bound on the genus of a curve, see e.g.~\cite{harris}. Although Castelnuovo's bound is best known in the case of smooth curves, it actually gives an upper bound on the \emph{arithmetic} genus of any non-degenerate integral singular curve, as a careful reading of loc.~cit.~shows.
We remark that our approach to proving Theorem~\ref{thm:1} is different from the approach of \cite{harris}.

\begin{table}
	\begin{tabular}{|l |l |l|l|l|}
		\hline
\rule{0pt}{3.5ex} Case ID & Singularity type & Schubert Variety &  Codimension  \\	
\hline
\rule{0pt}{3.5ex}{\bf 1.1} & Cusp & \tb & $n-1$  \\
\rule{0pt}{3.5ex}{\bf 1.2} & Node & \tb & $n-2$  \\
\hline
\rule{0pt}{3.5ex}{\bf 2.1.a} & $(3,4,5)$-cusp & \tbc & $2n-1$  \\
\rule{0pt}{3.5ex}{\bf 2.1.b}/{\bf 3.1.d} & $(2,5)$/$(2,7)$-cusp & \tbd & $2n-2$  \\
\rule{0pt}{3.5ex}{\bf 2.2.a}/{\bf 3.2.f} & \begin{tabular}{@{}l}Tacnode/Node with\\third order contact\end{tabular} & \tbd & $2n-3$  \\
\rule{0pt}{3.5ex}{\bf 2.2.b} & \begin{tabular}{@{}l}Cusp with\\smooth branch\end{tabular}  & \tbc & $2n-2$  \\
\rule{0pt}{3.5ex}{\bf 2.3} & \begin{tabular}{@{}l}Ordinary\\triple point\end{tabular}  & \tbc & $2n-3$  \\
&&&\\
\hline
\end{tabular}\\
\vspace{.5cm}
\caption{Singularities via Projections}\label{table:strat1}
\end{table}
\begin{table}
	\begin{tabular}{|l |l |l|l|l|}
		\hline
\rule{0pt}{3.5ex} Case ID & Singularity type & Schubert Variety &  Codimension  \\	
\hline
\rule{0pt}{3.5ex}{\bf 3.1.a} & $(4,5,6,7)$-cusp & \tbcd & $3n-1$  \\
\rule{0pt}{3.5ex}{\bf 3.1.b} & $(3,5,7)$-cusp & \tbce & $3n-2$  \\
\rule{0pt}{3.5ex}{\bf 3.1.c} & $(3,4)$-cusp & \tbcf & $3n-3$  \\
\rule{0pt}{3.5ex}{\bf 3.2.a} & \begin{tabular}{@{}l}$(3,4,5)$-cusp with\\smooth branch\end{tabular}  & \tbcd & $3n-2$  \\
\rule{0pt}{3.5ex}{\bf 3.2.b} & \begin{tabular}{@{}l}Rhamphoid cusp\\with smooth branch\end{tabular}  & \tbce & $3n-3$  \\
\rule{0pt}{3.5ex}{\bf 3.2.c} & \begin{tabular}{@{}l}Two independent\\cusps\end{tabular}  & \tbcd & $3n-2$  \\
\rule{0pt}{3.5ex}{\bf 3.2.d} & \begin{tabular}{@{}l}Cusp with collinear\\smooth branch\end{tabular}  & \tbce & $3n-3$  \\
\rule{0pt}{3.5ex}{\bf 3.2.e} & \begin{tabular}{@{}l}Cusp with coplanar\\smooth branch\end{tabular}  & \tbcf & $3n-4$  \\
\rule{0pt}{3.5ex}{\bf 3.3.a} & \begin{tabular}{@{}l}Cusp with $2$\\smooth branches\end{tabular}  & \tbcd & $3n-3$  \\
\rule{0pt}{3.5ex}{\bf 3.3.b} & \begin{tabular}{@{}l}Tacnode with extra\\smooth branch\end{tabular}  & \tbce & $3n-4$  \\
\rule{0pt}{3.5ex}{\bf 3.3.c} & \begin{tabular}{@{}l}Planar\\triple point\end{tabular}  & \tbcf & $3n-5$  \\
\rule{0pt}{3.5ex}{\bf 3.3.d} & \begin{tabular}{@{}l}Ordinary\\quadruple point\end{tabular}  & \tbcd & $3n-4$  \\
&&&\\
\hline
\end{tabular}\\
\vspace{.5cm}
\caption{Singularities via Projections (cont'd)}\label{table:strat2}
\end{table}

Inspired by classification results for plane curves, we use our approach via multifiltrations to classify singularities arising by projecting smooth curves 
away from low-dimensional linear spaces. We summarize our results from \S\ref{sec:project} in the following.
Let $X$ be a smooth projective curve of geometric genus $\rho_g$ with a very ample line bundle $\mcL$ of degree $d$. 
Thus $X$ may be viewed as embedded in $\PP(V)$, where $V=H^0(\mcL)^*$.

\begin{thm}[See \S\ref{sec:project}]\label{thm:2}
Let $\ell\in\NN$ with $2\ell<d-2\rho_g$. Set $n+1=\dim V-\ell$ and assume that $n>2$.
\begin{enumerate}
	\item\label{item:bir} Consider the open subvariety $U\subset G(\ell,V)$ consisting of those linear spaces $L$ such that the projection of $X$ with center $\PP(L)$ is birational and induced by a basepoint free linear system. Then $U$ consists exactly of those $L$ for which $\PP(L)$ does not meet $X$.
		\newcounter{myco}
		\setcounter{myco}{\value{enumi}}
\end{enumerate}
Assume now furthermore that $\ell\leq 3$:
\begin{enumerate}
		\setcounter{enumi}{\value{myco}}
	\item\label{item:mult} For $L\in U$, the singularities of the corresponding projection of $X\subset\PP(V)$ are determined by the intersection behaviour of $L$ with appropriate multifiltrations in $V=H^0(X,\mcL)$, with two exceptions recorded in Remark~\ref{rem:caveat}.
	
	\item\label{item:sing} The 21 types of singularities which can occur are listed in Tables~\ref{table:strat1} and~\ref{table:strat2}. The notation {\bf i.j.k} used there indicates a singularity with singularity degree {\bf i}, and {\bf j} many branches.
	\item\label{item:flags} For each singularity type, there is a locally closed subvariety of $U\subset G(\ell,V)$ for which the corresponding projections are exactly those admitting this singularity. This subvariety is a non-empty open subset of a {\bf j}-parameter family of Schubert varieties of type corresponding to the partition listed in the table, and is obtained by varying the underlying flag.
	\item\label{item:codim} The codimension of these subvarieties is listed in Tables~\ref{table:strat1} and~\ref{table:strat2}. For any tuple of singularity types, the corresponding locally closed subvariety of $U\subset G(\ell,V)$ parametrizing projections with exactly these singularities is non-empty if and only if the sum of the singularity degrees is at most $\ell$. 
In cases where it is non-empty, its codimension is the sum of the codimensions for each singularity type.
\end{enumerate}
\end{thm}

\noindent In particular, applying these results to projections of the rational normal curve, we obtain a complete classification of the configurations of singularities occurring for non-degenerate rational curves of degree $d$ whenever $d-n \leq 3$ and $d < 2n$.
\begin{rem}
The inequalities  $d-n \leq 3$ and $d < 2n$ for rational curves explain the title of this article. More generally, 
the condition $2\ell < d - 2\rho_g$ gives $d<2n-2 h^1(\mathcal L)$, recalling that $\ell + n + 1 = h^0(\mathcal L)$ and applying Riemann-Roch to 
$\mcL$. Similarly, the condition $\ell\leq 3$ gives $d-n\leq 3+\rho_g - h^1(\mcL)$. 
\end{rem}
\begin{rem}
	While there is a lack of agreed-upon terminology for singularity types, we hope that our descriptions in Tables~\ref{table:strat1} and~\ref{table:strat2} are suggestive. Precise descriptions may be found in \S\ref{ssec:sing}. We use the notation $(a_1,a_2,\ldots,a_k)$-cusp to describe any unibranched singularity whose valuation semigroup is generated by the elements $a_i\in\NN$. For example, a standard cusp is a $(2,3)$-cusp.

We use the conventions of e.g.~\cite[Chapter 4]{3264} for indexing Schubert varieties by partitions.
\end{rem}

\begin{ex}[Rational quintics in $\PP^3$]
Any rational quintic $C$ in $\PP^3$ is a basepoint-free projection of $X_5\subset \PP^5$, the rational normal curve of degree five.
Thus, we are in the situation of Theorem~\ref{thm:2} with $\ell=2$, $d=5$, $n=3$.
We obtain that the possible configurations of singularities that $C$ can have are either exactly one of the singularities of Table~\ref{table:strat1}, or two nodes, two cusps, or a node and a cusp. Note that since $\ell=2$, the case {\bf 2.1.b}/{\bf 3.1.d} can only be a $(2,5)$-cusp, and the case {\bf 2.2.a}/{\bf 3.2.f} can only be a tacnode.

The dimension of the Grassmannian parametrizing the projections is $\ell\cdot(n+1)=8$, and the dimensions of the strata corresponding to these configurations range from $8$ to $3$.
In each case, the results of Theorem~\ref{mainthmproject} can be used to construct a parametrization of a curve with given singularity configuration. For example, for a cusp with smooth branch ({\bf 2.2.b}), we know that $\PP(L)$ intersects a tangent line of $X_5$, as well as a secant line meeting this tangent line (in $X_5)$. After acting by $\textnormal{PGL}(2,\mathbb C)$ on $X_5\subset \PP^5$, we can assume that $\PP(L)$ is the span of 
\begin{align*}
	(c:1:0:0:0:0)\\ 
	(1:0:0:0:0:1)
\end{align*}
	for some $c\in\KK$. We obtain that $C$ is parametrized by
	\begin{align*}
		x^5+cxy^4-y^5,\quad x^2y^3,\quad x^3y^2,\quad\textrm{and}\ x^4y.
	\end{align*}

Similarly, up to the $\textnormal{PGL}(2,\mathbb C)$ action and the choice of coordinates on $\PP^3$, 
  there are exactly two rational space quintics with a singularity isomorphic to the $(3,4,5)$-cusp ({\bf 2.1.a}) --- their parametrizations are:
  \begin{align*}
	x^5, x^2y^3, xy^4, y^5;\qquad \textrm{or}\quad
	x^5+x^3y^2, x^2y^3, xy^4, y^5.
\end{align*}
\end{ex}
\begin{rem}
	The families of flags appearing in statement (\ref{item:flags}) of Theorem~\ref{thm:2} are in general \emph{not} families of osculating flags, but are rather obtained from our more complicated multifiltrations.
\end{rem}
\subsection{Related Work and Organization}
In recent work, Cotterill,~Feital,~and Martins have used valuation semigroups to study properties of singular rational curves, see \cite{ethan1,ethan2}. While we take a slightly different perspective, our approach is closely related to theirs. In particular, \cite[Theorem 2.1]{ethan1} is connected to our Theorem \ref{thm:2}, although neither result implies the other.

Additonally, several other techniques have been used to study rational curves in projective space. One such technique involves studying the syzygies among the functions parametrizing a rational curve \cite{syzygies}. Furthermore, there is significant literature dedicated to studying the splitting types of the normal and restricted tangent bundles on rational curves, see e.g.~\cite{kollar_book_rational_curves,clemens,alzati,coskun,kebekus_families_of_singular_rational_curves}.
Finally, the \emph{simple singularities} that can occur in curves have been studied in e.g.~\cite{af,stevens}.

We now describe the organization of the remainder of this paper.
In \S\ref{sec:gap}, we introduce \emph{gap functions}, which are a combinatorial tool for measuring to what extent a complete ring differs from its normalization. Our key lemma (Lemma~\ref{lemma:key}) gives a kind of bound on the ``regularity'' of a gap function and features prominently in the rest of our analysis.
We introduce our second main tool, the \emph{multifiltration} arising from a collection of osculating flags in \S\ref{sec:osc}, and relate it to our gap functions. We then use the relation between multifiltrations and gap functions  to prove Theorem~\ref{thm:1} in \S\ref{sec:bound}.

In \S\ref{sec:sing}, we offer a detailed study of small degree gap functions. This analysis is employed in \S\ref{sec:project} to classify all possible configurations of singularities arising from projections, proving Theorem~\ref{thm:2}.

\subsection*{Acknowledgements}
We are thankful to the IMPAN ``Rational Curves'' working group for helpful discussions, 
   especially to Weronika Buczy{\'n}ska.
   Jan Christophersen, John Christian Ottem, Ragni Piene, and Frank Sottile all provided useful comments. The anonymous referee helped us simplify our proof of Corollary \ref{thm:1new}.
E.~Ventura would additionally like to thank Edoardo Ballico, Riccardo Re, and Francesco Russo for useful discussions.

J.~Buczy{\'n}ski is supported by the National Science Center, Poland, project ``Complex contact manifolds and geometry of secants'', 2017/26/E/ST1/00231.
E.~Ventura would like to thank Simon Fraser University, where part of this project was conducted, for the hospitality.
N.~Ilten was partially supported by NSERC. 
All authors were partially supported by the grant 346300 for IMPAN from the Simons Foundation and the matching 2015-2019 Polish MNiSW fund.
Finally, the paper is also a part of the activities of AGATES research group.

\section{Gap Functions}\label{sec:gap}
\subsection{Preliminaries}\label{sec:prelim}
For some $r\in \NN$, consider the ring
\[
	S=\prod_{i=1}^r \KK[[t_i]].
\]
\begin{defn}\label{defn:gap_function}
	For any $\KK$-vector space $R\subset S$, the \emph{gap function} of $R$ in $S$ is the map
\[\lambda_R:\ZZ_{\geq 0}^r\to \ZZ_{\geq 0}\]
with
\[
	\lambda_R(\alpha)=\dim S/\left(R+\langle t_1^{\alpha_1},\ldots,t_r^{\alpha_r}\rangle\right),
\]
where $\langle \bullet \rangle$ denotes the ideal in $S$ generated by $\bullet$, 
    while the quotient $/$ is of vector spaces.
Note that $t_i^0$ is not the multiplicative unit of $S$, but rather the $r$-tuple with $1$ at the $i$th position and zero elsewhere.
We will simply write $\lambda=\lambda_R$ whenever it is clear what $R$ is.
\end{defn}

\begin{ex}
   Let $C$ be a curve and $Q \in C$ a closed point. 
   Consider the local ring $\CO_{C,Q}$, let $R$ be its completion,
     and $S\simeq \prod_{i=1}^r \KK[[t_i]]$ the normalization of $R$. Here $r$ is the number of branches of the singularity at $Q$.
   Thus $\lambda_R$ is an invariant of the singularity $(C,Q)$ and in the cases considered in this article 
     (see Proposition~\ref{prop:sing}) the gap function  is sufficient to determine the singularity type.
   That is, $\lambda_R$ determines $R \subset S$ up to an automorphism of $S$.
\end{ex}
\noindent In \S\ref{sec:rsing} we discuss further examples of interesting $R \subset S$ 
   for which the gap function $\lambda_R$ is relevant to our investigations.
This includes the case when $R\subset S$ is a finite dimensional $\KK$-vector subspace determined by a linear system.

We introduce some further useful pieces of notation.
The ring $S$ is equipped with $r$ discrete valuations $\val_i:S\to \ZZ\cup\{\infty\}$ obtained by projecting $S$ to its $i$th factor and taking the standard discrete valuation on $\KK[[t_i]]$ given by the order of vanishing in $t_i$.
Composing these with the inclusion of $R$ in $S$, we obtain a map
\[
	\val:R\to (\ZZ\cup\{\infty\})^r
\]
sending $f\in R$ to $(\val_1(f),\ldots,\val_r(f))$.
We will denote the image of this map by $\Sigma=\val(R)$. If $R$ itself is a ring, then $\Sigma$ is a semigroup.
We call 
\[
	\delta(\lambda)=\sup_{\alpha\in\ZZ_{\geq 0}^r} \lambda(\alpha)
\]
the \emph{degree} of the gap function $\lambda$.
The relationship between $\Sigma$ and $R$ has been studied by a number of authors, see e.g.~\cite{barucci,carvalho,ethan1}.

Given an element $\alpha\in \ZZ_{\geq 0}^r$ and an index $1\leq i\leq r$, $\alpha_i$ denotes the $i$-th coordinate of 
   $\alpha$.
We set \[
|\alpha|=\sum \alpha_i.
\]
The element $e_i\in \ZZ^r$ denotes the $i$-th vector of the standard basis. 
\begin{defn}
	Given some element $\alpha\in \ZZ_{\geq 0}^r$ and an index $1\leq i \leq r$, we will say that \emph{$\alpha[i]$ belongs to} or \emph{is contained in} $\Sigma$ if there exists an element $\alpha'\in\Sigma$ such that $\alpha_j\leq \alpha_j'$ for all $1\leq j \leq r$, and $\alpha_i=\alpha_i'$. We will also refer to $\alpha[i]$ as an \emph{element} of $\Sigma$.
\end{defn}

\begin{rem}\label{rem:sigma}
It is straightforward to check that $\lambda(\alpha+e_i)-1\leq \lambda(\alpha)\leq \lambda(\alpha+e_i)$, and $\lambda(\alpha)=\lambda(\alpha+e_i)$ if and only if $\alpha[i]$ belongs to $\Sigma$.
In particular, $\alpha\in\Sigma$ if and only if $\lambda(\alpha)=\lambda(\alpha+e_i)$ for all $i$.
\end{rem}

We call a gap function $\lambda=\lambda_R:\ZZ_{\geq 0}^r\to \ZZ_{\geq 0}$ \emph{standard} if $\lambda(e_i)=0$ for all $i$, and $\lambda(e_1+\ldots+e_r)=r-1$.
Notice that the condition $\lambda(e_i)=0$ is satisfied if $R$ contains a unit of $S$. For a standard gap function $\lambda=\lambda_R$, its values are determined by its restriction to $\NN^r$. Indeed, it is straightforward to check that for any $\alpha\in\ZZ_{\geq 0}^r$ with $\alpha_i=0$, $\alpha\neq 0$, we have $\lambda(\alpha)=\lambda(\alpha+e_i)-1$.

\subsection{The Key Lemma}\label{sec:key}

The following lemma is the key result which will often allow us to calculate the singularity degree of a singular point on a curve $C$:
\begin{lemma}[Key Lemma]\label{lemma:key}
Fix $\gamma\in \ZZ_{\geq 0}$. Assume that $R \subset S$ is a subalgebra and let $\lambda=\lambda_R$.
Suppose that for any $\alpha\in\ZZ_{\geq 0}^r$ satisfying $|\alpha|\leq 2\gamma+2$, we have $\lambda(\alpha)\leq \gamma$. Then
\[
\lambda(\alpha)\leq \gamma
\]
for \emph{all} $\alpha\in\ZZ_{\geq 0}^r$, that is, the degree of $\lambda$ is at most $\gamma$.
\end{lemma}

We are going to repeatedly use the following observation throughout this section:

\begin{rem}\label{heartproof}
	Let $\alpha\in \ZZ_{\geq 0}^r$, $\alpha\neq 0$. Then $\lambda(\alpha)\leq \gamma$ if and only if there exists some  $i$ such that $\alpha-e_i\in \ZZ_{\geq 0}^r$, and either 
\begin{enumerate}
\item[(i)] $\lambda(\alpha-e_i)<\gamma$, or
\item[(ii)] $\lambda(\alpha-e_i)\leq \gamma$ and $\alpha[i]$ belongs to $\Sigma$. 
\end{enumerate}
Equivalently, $\lambda(\alpha)\leq \gamma$ if and only if for \emph{all} $i$ such that $\alpha-e_i\in \ZZ_{\geq 0}^r$, one of these two conditions hold.
This follows directly from Remark~\ref{rem:sigma}.
\end{rem}

The following construction will be used in the proof of Lemma \ref{lemma:key}
\begin{con}\label{con}
	Consider $b^1\in\ZZ_{\geq 0}^r$ and $c^0\in b^1+\ZZ_{\geq 0}^r$.
We will inductively construct two paths of lattice points connecting $b^1$ and $c^0$, with each step differing by a standard basis element of $\ZZ^r$. 
	Without loss of generality, we will reorder the coordinates of $\ZZ^r$ so that $(c^0-b^1)_i$ is odd if and only if $1\leq i \leq p$ for some $p\leq r$.

	To begin, we inductively set
\[
b^{i+1} = b^i + e_{j_i},\qquad c^i = c^{i-1}-e_{j_i}
\]
where $j_i$ is the smallest index such that $(c^{i-1}-b^i)_{j_i}\geq 2$, i.e.~the $j_i$-th coordinate of the lattice point $c^{i-1}-b^i$ is at least $2$. This procedure continues until the latter condition is violated for all coordinates. 

Let $\xi$ be the maximum integer such that $b^{\xi}$ is defined. We thus have sequences
\begin{align*}
&b^1,b^2,\ldots,b^{\xi-1},b^\xi\\
c^0,&c^1,\ldots,\ \ c^{\xi-1}
\end{align*}
with subsequent entries differing by a standard basis vector. By construction, we have that $c^{\xi-1}-b^\xi$ has $1$ as its first $p$ coordinates, and $0$ as its remaining coordinates.

We now set
\begin{align*}
&f^0=b^\xi,f^1=f^0+e_1,\ldots,f^i=f^{i-1}+e_i,\ldots,f^p=f^{p-1}+e_p=c^{\xi-1}\\
&g^0=b^\xi,g^1=g^0+e_p,\ldots,g^i=g^{i-1}+e_{p-i+1},\ldots,g^p=g^{p-1}+e_1=c^{\xi-1}.
\end{align*}
We thus obtain two sequences, whose subsequent elements are always differing by a standard basis vector:
\begin{align*}
(b^1,b^2,\ldots,b^\xi=f^0,f^1,\ldots,f^p=c^{\xi-1},c^{\xi-2},\ldots,c^0)\\
(b^1,b^2,\ldots,b^\xi=g^0,g^1,\ldots,g^p=c^{\xi-1},c^{\xi-2},\ldots,c^0)\\
\end{align*}
By comparing the value of $\lambda$ on two neighboring elements in these sequences, we obtain a number of potential elements belonging to $\Sigma$:
the first sequence leads to the elements 
\begin{equation}\label{eqn:s1}
b^1[j_1], b^2[j_2], \ldots, b^{\xi-1}[j_{\xi-1}],f^0[1],f^1[2],\ldots,f^{p-1}[p],c^{\xi-1}[j_{\xi-1}],\ldots,c^1[j_1]
\end{equation}
potentially belonging to $\Sigma$, while the second sequence leads to the potential elements
\begin{equation}\label{eqn:s2}
b^1[j_1], b^2[j_2], \ldots, b^{\xi-1}[j_{\xi-1}],g^0[p],g^1[p-1],\ldots,g^{p-1}[1],c^{\xi-1}[j_{\xi-1}],\ldots,c^1[j_1].
\end{equation}
\end{con}
In conjunction with the following, the reader may wish to consider Example~\ref{ex:key},  which illustrates our use of the construction above.

\begin{proof}[Proof of Lemma \ref{lemma:key}]
	We prove the lemma by induction on $|\alpha|$. Consider any $\alpha\in\ZZ_{\geq 0}^r$ such that $|\alpha|>2\gamma+2$ and for all $\alpha'$ with $|\alpha'|<|\alpha|$, $\lambda(\alpha')\leq \gamma$.
This implies that $\lambda(\alpha)\leq \gamma+1$. We will show that $\lambda(\alpha)\leq \gamma$.

	We will apply Construction \ref{con} to the situation $b^1=0$, $c^0=\alpha$, and consider the potential elements \eqref{eqn:s1} and \eqref{eqn:s2} of $\Sigma$.
We now notice by adding semigroup elements that
\begin{align}\label{addingsemigroupelements1}
b^i[j_i],c^i [j_i]\in \Sigma \implies(\alpha-e_{j_i})[j_i]\in \Sigma \qquad 1\leq i \leq \xi-1;
\end{align}
\begin{align}\label{addingsemigroupelements2}
f^i[i+1],g^{p-1-i}[i+1]\in\Sigma\implies (\alpha-e_{i+1})[i+1]\in \Sigma \qquad 0 \leq i \leq p-1.
\end{align}
By Remark~\ref{heartproof} and the assumption on $\alpha$, having one of these elements on the right hand side belong to $\Sigma$ implies that $\lambda(\alpha)\leq \gamma$.

We now count how many of the elements of \eqref{eqn:s1} and \eqref{eqn:s2} must actually belong to $\Sigma$. 
Since $\lambda(c^0)\leq \gamma+1$ by assumption and $\lambda(b^1)=0$, $\lambda$ increases at most $\gamma+1$ times, so at most $\gamma+1$ of the elements of \eqref{eqn:s1} (and similarly \eqref{eqn:s2}) do not belong to $\Sigma$.
In other words, $|\alpha|-(\gamma+1)$ of the elements of \eqref{eqn:s1}, as well as $|\alpha|-(\gamma+1)$ of the elements of \eqref{eqn:s2}, must belong to $\Sigma$.
Thus, the multiset consisting of \eqref{eqn:s1} and \eqref{eqn:s2} has $2|\alpha|$ elements, and at least $2(|\alpha|-(\gamma+1))$ elements in $\Sigma$.
There are $|\alpha|$ pairs of elements, which, if belonging to $\Sigma$, imply $\lambda(\alpha)\leq \gamma$ (we are counting $b^i[j_i],c^i [j_i]$ twice).
But 
\[
2(|\alpha|-(\gamma+1))>|\alpha|>0
\]
since $|\alpha|>2\gamma+2$. So by the pigeonhole principle, we conclude that one such pair belongs to $\Sigma$, and $\lambda(\alpha)\leq \gamma$.

The claim of the lemma follows by induction and the hypothesis on $\lambda$.
\end{proof}

\begin{ex}\label{ex:key}
	We consider the situation $r=2$, $\gamma=3$ from Lemma~\ref{lemma:key}, and assume that the hypotheses of the lemma are fulfilled. The point $\alpha=(5,4)$ has $|\alpha|=9>2\gamma+2$. We will show that nonetheless $\lambda(\alpha)\leq 3$.

	For this, we set
	\begin{align*}
b^1=(0,0)\qquad b^2=(1,0)\qquad b^3=(2,0)\qquad b^4=(2,1)\qquad b^5=(2,2) \\
c^0=(5,4)\qquad c^1=(4,4) \qquad c^2=(3,4)\qquad c^3=(3,3)\qquad c^4=(3,2).
	\end{align*}
	See Figure~\ref{fig:key}.
Along the path $b^1,b^2,b^3,b^4,b^5,c^4,c^3,c^2,c^1,c^0$ $\lambda$ starts with value $0$, ends with value at most $\gamma+1=4$, and is non-decreasing.	
Thus, in the nine steps in this path, $\lambda$ is constant at least five times. This means that five of 
\begin{align*}
	(0,0)[1],\quad (1,0)[1],\quad	(2,0)[2],\quad (2,1)[2],\quad (2,2)[1],\\ \quad (3,2)[2], \quad (3,3)[2], \quad (3,4)[1], \quad (4,4)[1]
\end{align*}
must belong to $\Sigma$, and hence one of the pairs
\begin{align*}	
\{(0,0)[1],
(4,4)[1]\},\quad \{(1,0)[1],(3,4)[1]\}, \quad \{(2,0)[2],(3,3)[2]\}\\
\{(2,1)[2],(3,2)[2]\}, \quad \{(2,2)[1],(2,2)[1]\}\\
\end{align*}
must also be contained in $\Sigma$. But this implies by adding semigroup elements that either $(4,4)[1]$ or $(5,3)[2]$ is in $\Sigma$, which by Remark~\ref{heartproof}
shows that $\lambda(\alpha)\leq 3$.

We note here that this situation is slightly simpler than the general situation in the proof of Lemma~\ref{lemma:key}: since $r=2$, the two different paths appearing in the proof end up equal to one another, which is why we only see a single path in this example.
\end{ex}

\begin{figure}
	\begin{tikzpicture}
		\draw[very thin,color=gray] (0,0) grid (5,4);
		\draw[very thick] (0,0) -- (2,0) -- (2,2) -- (3,2) -- (3,4) -- (5,4);
		\draw[fill,color=red] (0,0) circle [radius=.1];
		\draw[fill,color=blue] (1,0) circle [radius=.1];
		\draw[fill,color=green] (2,0) circle [radius=.1];
		\draw[fill,color=yellow] (2,1) circle [radius=.1];
		\draw[fill,color =orange] (2,2) circle [radius=.1];
		\draw[fill,color = yellow] (3,2) circle [radius=.1];
		\draw[fill,color=green] (3,3) circle [radius=.1];
		\draw[fill,color=blue] (3,4) circle [radius=.1];
		\draw[fill] (5,4) circle [radius=.1];
		\draw[color =orange,very thick ] (4,4) circle [radius=.12];
		\draw[fill,color=blue] (4,4) circle [radius=.1];
		\draw[fill,color =red,very thick ] (4,4) circle [radius=.05];
		\draw[color =green,very thick ] (5,3) circle [radius=.1];
		\draw[color =yellow,very thick ] (5,3) circle [radius=.05];
		\draw (0,0) node [above left] {$b^1$};
		\draw (1,0) node [below right] {$b^2$};
		\draw (2,0) node [below right] {$b^3$};
		\draw (2,1) node [above left] {$b^4$};
		\draw (2,2) node [above left] {$b^5$};
		\draw (3,2) node [below right] {$c^4$};
		\draw (3,3) node [above left] {$c^3$};
		\draw (3,4) node [above left] {$c^2$};
		\draw (4,4) node [below right] {$c^1$};
		\draw (5,4) node [below right] {$c^0$};
		\draw (5,4) node [above right] {$\alpha$};
		\draw (4,4) node [above ] {${b^5+b^5}$};
		\draw (4,4.3) node [above ] {${b^2+c^2}$};
		\draw (4,4.6) node [above ] {${b^1+c^1}$};
		\draw (5,3) node [below right] {$b^3+c^3$};
		\draw (5,2.6) node [below right] {$b^4+c^4$};
	\end{tikzpicture}

\caption{The path from Example~\ref{ex:key}\label{fig:key}}
\end{figure}
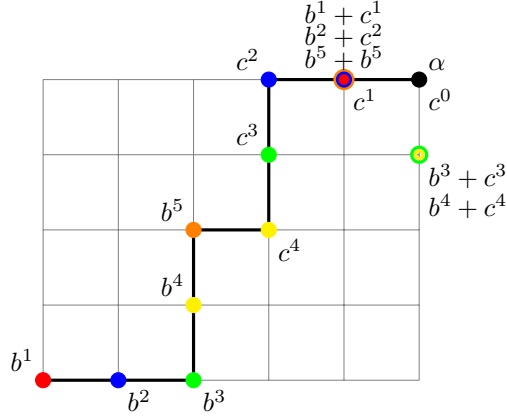

When we additionally assume that $\lambda=\lambda_R$ is a standard gap function, we have a slightly modified result, which we will use in \S\ref{sec:sing}.

\begin{lemma}\label{lemma:help}
	Let $R\subset S$ be a subalgebra and assume that $\lambda=\lambda_R$ is a standard gap function.
	Fix $\gamma\geq r-1$.
Consider any $\alpha\in \NN^r$ with $\ell$ coordinates $\alpha_i$ equal to one, such that $|\alpha|>2\gamma+2-\ell$. Assume that for all $\alpha'\in \NN^r$ with $\alpha'\neq \alpha$ and $(\alpha-\alpha')_i\geq 0$ for all $i$, we have $\lambda(\alpha')\leq \gamma$. Then $\lambda(\alpha)\leq \gamma$. 
\end{lemma}
\begin{proof}
Without loss of generality, we will reorder the coordinates of $\alpha$ so that $\alpha_1,\dotsc,\alpha_p$ are odd and larger than $1$, $\alpha_{p+1},\ldots,\alpha_{q}$ are even, and $\alpha_{q+1},\ldots,\alpha_r$ all equal $1$. 
Notice that $\ell=r-q<r$.

Similar to the proof of Lemma \ref{lemma:key}, we will apply Construction \ref{con}, this time to
\[
b^1 = (1,\ldots,1)\qquad c^0 = \alpha-\sum_{i=1}^q e_i.\]
Our reordering of the coordinates above agrees with the reordering in Construction \ref{con}. The construction leads to potential elements of $\Sigma$ in \eqref{eqn:s1} and \eqref{eqn:s2}.

As before, by adding semigroup elements, we derive \eqref{addingsemigroupelements1} and \eqref{addingsemigroupelements2}.
By Remark~\ref{heartproof}, having one of these elements on the right hand side belong to $\Sigma$ implies $\lambda(\alpha)\leq \gamma$.

We now count how many of the elements of \eqref{eqn:s1} and \eqref{eqn:s2} must actually belong to $\Sigma$. 
Since $\lambda(c^0)\leq \gamma$ by assumption,
$\lambda$ increases at most $\gamma-(r-1)$ times along the path from $b^1$ to $c^0$. Remark~\ref{heartproof} then implies that
$(|c^0|-r)-(\gamma-(r-1))=|\alpha|-q-(\gamma+1)$ of the elements of \eqref{eqn:s1}, as well as $|\alpha|-q-(\gamma+1)$ of the elements of \eqref{eqn:s2}, must belong to $\Sigma$.
Thus, from the multiset consisting of \eqref{eqn:s1} and \eqref{eqn:s2},
we count at least $2\left(|\alpha|-q-(\gamma+1)\right)$ elements in $\Sigma$.
This multiset partitions into $|\alpha|-r-q$ pairs of elements, each of which, if belonging to $\Sigma$, implies $\lambda(\alpha)\leq \gamma$.
But 
\[
2(|\alpha|-q-(\gamma+1))>|\alpha|-r-q>0
\]
since $|\alpha|>2\gamma+2-(r-q)$, and $\gamma\geq r-1$. So by the pigeonhole principle, one such pair belongs to $\Sigma$, and we are done.
\end{proof}

\subsection{Convergence}\label{sec:convergence}
The ring $S$, being a product of power series rings, is complete with respect to the topology induced by the ideal generated by $t_1,\ldots,t_r$; see e.g. \cite[Chapter 7]{eisenbudCommAlg}. We say that $R\subset S$ is a {\it complete subalgebra} of $S$ if $R$ is a subalgebra of $S$ which is complete in the same topology as $S$. In this situation, we show the following. 

\begin{lemma}\label{lemma:convergence}
Assume that $R$ is a complete subalgebra of $S$, that $\lambda=\lambda_R$ is a standard gap function, and that $\dim S/R$ is finite.
Fix some $1\leq i \leq r$. Consider $\alpha\in \NN^r$ such that for any $\alpha'\in \NN^r$ with $\alpha'_i=\alpha_i$ and $\alpha'_j\ge \alpha_j$ for all $j\ne i$,
\[
\lambda(\alpha'+e_i)=\lambda(\alpha').\]
Then for some unit $u\in \KK[[t_i]]$,  $u\cdot t_i^{\alpha_i}\in R$.
\end{lemma}
\begin{proof}

Since $\dim S/R$ is finite, the function $\lambda$ has a maximum value $\gamma$. Let $b\in \NN^r$ be such that $\lambda(b)=\gamma$. Then for any element $b'\in b+\ZZ_{\geq 0}^r$, we know by Remark~\ref{rem:sigma} that for any $j$, $b'[j]$ is in $\Sigma$. 

We inductively construct a sequence of elements in $R$ converging (in $S$, and thus in $R$) to $t_i^{\alpha_i}$ times a unit. 
Choose $\alpha'\in \NN^r$ with $\alpha'\geq \alpha$, $\alpha'_i=\alpha_i$, and $\alpha'_j\geq b_j$ for all $j\neq i$. By our hypothesis on $\alpha$, we know that $\alpha'[i]\in \Sigma$. Thus, there is an element $s_0\in R$ with $\val_i(s_0)=\alpha_i$ and $\val_j(s_0)\geq \alpha'_j\geq b_j$ for $j\neq i$.

Suppose given $m\in \NN^r$ with $m_i\geq \alpha_i$ (for our fixed $i$), we have constructed an element $s$ of $R$ with $\val_i(s)=\alpha_i$ and $\val_j(s)\geq \max\{m_j,b_j\}$ for all $j\neq i$.
Fix some $k\neq i$. By our choice of $b$, and Remark~\ref{rem:sigma}, there exists $f\in R$ with $\val_k(f)=\val_k(s)$, $\val_i(f)>m_i$, and with $\val_j(f)>\val_j(s)$ for $j\neq i,k$. 

Furthermore, there exists a constant $c\in \KK$ such that 
\[s'=s+cf\]
has $\val_k(s')>\val_k(s)$.
Starting with $s=s_0$ and iteratively replacing $s$ by $s'$ as we vary $k$ leads to a convergent sequence. Its limit $t$ has $\val_i(t)=\alpha_i$, and $\val_j(t)=\infty$ for $j\neq i$; the claim of the lemma follows. 
\end{proof}

\begin{lemma}\label{lemma:d2}
Assume that $R$ is a complete subalgebra of $S$, that $\dim S/R$ is finite, and that $\lambda=\lambda_R$ is a standard gap function. Set 
\[
\delta_{S/R} =\dim S/R.
\]
Then there exists $\alpha\in \NN^r$ with $|\alpha|\leq 2\delta_{S/R} $
such that $\lambda(\alpha)= \delta_{S/R} $.
In particular, $\delta(\lambda)=\delta_{S/R}$. Furthermore, $t_i^{k}\in R$ for all $k\geq 2\delta_{S/R}$.
\end{lemma}
\begin{proof}
Since $\dim S/R$ is finite, $\lambda$ has a maximum $\gamma$. So there exists some $\alpha\in \NN^r$ with $\lambda(\alpha)=\gamma$. Then $\alpha$, or any translate by $\ZZ^r_{\geq 0}$, satisfies the requirements for  Lemma~\ref{lemma:convergence}.

We then conclude that for $k\geq \alpha_i$, $u_k\cdot t_i^k\in R$ for some unit $u_k$ of $S$. Now, by multiplying by appropriate constants, we may construct sequences converging to $t_i^k$ for any $k\geq \alpha_i$.
This in turn implies that $\lambda(\alpha)=\delta_{S/R}$,
since 
\[
	S/R=S/(R+\langle t_1^{\alpha_1},\ldots,t_r^{\alpha_r}\rangle).
\]

We show we can choose $\alpha$ as in the statement. The assumption on $\lambda$ implies $\delta_{S/R}\geq r-1$.
By contradiction, suppose that no such $\alpha$ exists. Then $\lambda(\alpha)\leq \delta_{S/R}-1$ for all $\alpha\in \NN^r$ satisfying $|\alpha|\leq 2\delta_{S/R}$. Thus, taking $\gamma=\delta_{S/R}-1$ in Key Lemma~\ref{lemma:help}, we obtain $\lambda(\alpha)\leq \delta_{S/R}-1$ for all $\alpha$. However, this is impossible in view of the first part of this proof. The final claim of the lemma now also follows from the first part of this proof.
\end{proof}

\section{Curves, Osculating Flags, and Multifiltrations}\label{sec:osc}
\subsection{Constructing the Multifiltrations}
Let $X$ be a smooth projective curve of geometric genus $\rho_g$, and $\mcL$ a line bundle on $X$ of degree $d\geq 1$. Set $W=H^0(X,\mcL)$.
For any $0\leq i\leq d+1$, and any point $P\in X$, we let $W^i(P)\subset W$ be the vector space of sections of $\mcL$ which vanish to order at least $i$ at $P$. In other words, we may think of $W^i(P)$ as the subspace 
\[
W^i(P)=H^0(X,\mcL(-i\cdot P))\subseteq H^0(X,\mcL)=W.
\]
Dually, we set $V=W^*$ and take 
\[V^i(P)=(W^i(P))^\perp\subseteq V.\]
We thus obtain filtrations
\begin{align*}
	W=W^0(P)\supseteq W^{1}(P)\supseteq\ldots\supseteq W^{d+1}(P)= 0\\
	0=V^0(P)\subseteq V^{1}(P)\subseteq\ldots\subseteq V^{d+1}(P)= V
\end{align*}
of $W$ and $V$. 
For $i\leq 0$ or $i\geq d+1$, we set $V^i(P)=0$ and $V^i(P)=V$, respectively.

It is straightforward to see that
\[
\dim W^i(P)\leq \dim W^{i-1}(P)  \leq \dim W^i(P)+1
\]
for any $1\leq i \leq d$. By the theorem of Riemann-Roch, 
\[
\dim W\geq d+1-\rho_g
\]
so there are at most $\rho_g$ indices $1\leq i \leq d+1$ for which $\dim W^i(P)=\dim W^{i-1}(P)$. 
Dually,  
\[
\dim V^i(P)\leq \dim V^{i+1}(P)  \leq \dim V^i(P)+1
\]
for any $0\leq i \leq d$, and there are at most $\rho_g$ indices $0\leq i \leq d$ for which $\dim V^i(P)=\dim V^{i+1}(P)$. In particular,
\begin{equation}\label{eqn:dim1}
i-\rho_g \leq \dim V^i(P)\leq i
\qquad 0\leq i \leq d+1. 
\end{equation}
\begin{rem}[Ramification indices and osculating spaces]\label{rem:osc}
Given $P\in X$, one may consider the set
\[
\Rf_P=	\{\ord_P(s)\ |\ s\in W\}=\{i\ |W^i\neq W^{i+1}\}.
\]
This determines a non-decreasing sequence $0\leq\ram_0(P)\leq \ram_1(P)\leq\ldots\leq \ram_{\dim W-1}(P)$ by the formula
$\Rf_P=\{i+\ram_i(P)\}.$ The $\ram_i(P)$ are called the ramification indices of $\mcL$ at $P$, and $\dim V^{i+\ram_i(P)}(P)=i$. The image of $V^{i+\ram_i(P)}(P)$ in $\PP(V)$ is often called the osculating $(i-1)$st plane of the image of $X$ under the complete linear system $W\subset H^0(X,\mcL)$.
See e.g.~\cite{eh1} for more details.  
\end{rem}
\begin{ex}
	When $X=\PP^1$, \eqref{eqn:dim1} implies that $\dim V^i(P)=i$ for $0\leq i \leq d+1$.
\end{ex}

Consider now $r$ distinct points $P_1,\ldots,P_r\in X$. These points determine a $\ZZ^r$-graded multifiltration $\F^{\bullet}(P_1,\ldots,P_r)$ of $V$, where for $\alpha=(\alpha_1,\ldots,\alpha_r)\in\ZZ^r$ we set
\[
	\F^{\alpha}=\F^{\alpha}(P_1,\ldots,P_r)=\langle V^{\alpha_1}(P_1),\ldots,V^{\alpha_r}(P_r)\rangle.
\]
Recall that for $\alpha\in \ZZ^r$, we set $|\alpha|=\sum_{i=1}^r \alpha_i$.
\begin{lemma}\label{lemma:bound}
	For $\alpha\in \ZZ_{\geq 0}^r$ with $|\alpha|\leq d+1$, we have
\[
	|\alpha|-\rho_g\leq \dim \F^{\alpha}\leq |\alpha|.
\]
Furthermore, if $|\alpha|\leq d+1-2\rho_g$, then 
\[
	\dim \F^{\alpha}= |\alpha|.
\]

\end{lemma}
\begin{proof}
	The upper bound on $\dim \F^{\alpha}$ is immediate from the upper bound of \eqref{eqn:dim1}. For the lower bound, we proceed as follows.
	Observe that
\[
	\F^{\alpha}=\left(\bigcap_{i=1}^r W^{\alpha_i}(P_i)\right)^\perp=H^0\left(X,\mcL\left(-\sum \alpha_i P_i\right)\right)^\perp.
\]
Since
\[	\dim H^0\left(X,\mcL\left(-\sum \alpha_i P_i\right)\right)\leq d+1-|\alpha|\]
and $\dim W\geq d+1-\rho_g$, we obtain
\begin{align*}
\dim \F^\alpha=\dim W-
	\dim H^0\left(X,\mcL\left(-\sum \alpha_i P_i\right)\right)\\
\geq (d+1-\rho_g)-(d+1-|\alpha|)\\=|\alpha|-\rho_g.
\end{align*}
This gives the first inequality.

Assuming additionally that $|\alpha|\leq d+1-2\rho_g$, we obtain that
\[	\dim H^0\left(X,\mcL\left(-\sum \alpha_i P_i\right)\right)= d+1-|\alpha|-\rho_g\]
and the desired equality follows.
\end{proof}

\subsection{Relation to Singularities}\label{sec:rsing}
Fix some proper subspace $L$ of $V$ of codimension $n+1\geq 3$. Dually, we have a linear system $M:=L^\perp\subseteq W=H^0(X,\mcL)$ giving rise to a (non-constant) morphism 
\[
\phi \colon X\to \PP^n
\]
whose image is $C\subset \PP^n$. We will make the strong assumption that $\phi\colon X \to C$ is birational.
Furthermore,
we will assume that the linear system $M$ is basepoint free. This will be the case if e.g.~$\mcL\cong \phi^*(\CO_{\PP^n}(1))$, which in particular implies that $\deg \mcL=d=\deg C$.
We are interested in understanding the singularities of $C$.

Fix  points $Q_1,\ldots,Q_m$ of $C$, and let $\phi^{-1}(Q_i)$ consist of $P_{i1},\ldots,P_{ir_i}\in X$. Here, $r_i$ is the number of branches of the curve $C$ at $Q_i$. 
Let $R_i=\widehat\CO_{C,Q_i}$ be the completion of the local ring of the curve $C$ at $Q_i$. The ring $R_i$ sits inside its normalization, which coincides with the product $S_i$ of the completions of the local rings of $X$ at $P_{ij}$:
\[
	R_i=\widehat\CO_{C,Q_i}\hookrightarrow S=\prod_{j=1}^{r_i} \widehat\CO_{X,P_{ij}}\cong \prod_{j=1}^{r_i}\KK[ [t_{ij}]].
\]
The \emph{singularity degree} of $C$ at $Q_i$ is
\[
\delta(Q_i):=\dim S_i/R_i.
\]
Note that $\delta(Q_i)=0$ if and only if $Q_i$ is a smooth point of $C$. 
Letting $\rho_a$ denote the arithmetic genus of $C$, we then have
\[
	\rho_a=\rho_g+\sum_{Q\in C} \delta(Q),
\]
see e.g.~\cite[Exercise III.1.8]{hartshorne} and \cite[Lemma 3]{hironaka}.

Let $s\in M\subset H^0(X,\mcL)$ be a section which does not vanish at any $P_{ij}$ for $i=1,\ldots,m$, $j=1,\ldots,r_i$. This exists since the linear system $M$ is basepoint free by our assumptions. 
Then the $\KK$-vector space
\[
	R':=\frac{1}{s}\cdot M
\]
sits inside $\CO_{C,Q_i}$ for each $i$, and hence inside of $R_i$ and $S_i$.
We may thus view $R'=(1/s)M$ as a subspace of 
\[
	S=\prod_{i=1}^m S_i=\prod_{i=1}^m\prod_{j=1}^{r_i} \KK[[t_{ij}]].
\]
For \[\alpha=(\alpha_{ij})_{\substack{1\leq i \leq m\\1\leq j \leq r_i}}\in \ZZ_{\geq 0}^{r_1}\times\ldots\times\ZZ_{\geq 0}^{r_m}\] we set
\[
	\mu(\alpha)=\dim R'/\left(R'\cap\big\langle t_{ij}^{\alpha_{ij}}\ |\ 1\leq i \leq m,\ 1\leq j\leq r_i\big\rangle\right).
\]
Here 
$\langle t_{ij}^{\alpha_{ij}}\rangle$ denotes the ideal of $\prod S_i$ generated by $t_{ij}^{\alpha_{ij}}$ 
  (viewed as an element of the subring $\KK[[t_{ij}]]$) for $1\leq i \leq m$, $1\leq j \leq r_i$.
Note as in Definition~\ref{defn:gap_function} that the element $t_{ij}^0$ is not the multiplicative unit of 
	$\prod_{i=1}^m S_i$ (unless $m=r_1=1$).
We are viewing $R'=(1/s)M$ as a sub-vector space of $\prod S_i$, and the quotient is taken in the category of $\KK$-vector spaces.

We now consider the multifiltration $\F^{\bullet}=\F^{\bullet}(P_{11},\ldots,P_{mr_m})$ of $V$.
\begin{lemma}\label{lemma:mu}
	For $\alpha\in\ZZ_{\geq 0}^{r_1}\times\ldots\times\ZZ_{\geq 0}^{r_m}$ we have
	\[
\mu(\alpha)=\dim\F^\alpha-\dim (L\cap \F^\alpha).
	\]
\end{lemma}
\begin{proof}
We compute
	\begin{align*}
\dim(L\cap \F^\alpha)
&=\dim V-\dim ( (L\cap \F^\alpha)^\perp)\\
&=\dim V-\dim (M+(\F^\alpha)^\perp)\\
&=\dim V-\dim ((\F^\alpha)^\perp)- \dim (M)
+\dim (M\cap(\F^\alpha)^\perp)\\
&=\dim (\F^\alpha)- \dim (M)
+\dim (M\cap(\F^\alpha)^\perp).
\end{align*}
On the other hand,
\[
	\mu(\alpha)=\dim M-\dim \left(\frac{1}{s}M\cap \langle t_{ij}^{\alpha_{ij}}\rangle\right).
\]
But 
\[
	\left(\frac{1}{s}M\cap \langle t_{ij}^{\alpha_{ij}}\rangle\right)=
\frac{1}{s}\left(M\cap(\F^\alpha)^\perp\right).\]
Indeed, the condition that a section $f\in W$ of $\mcL$ vanish to order $\alpha_{ij}$ at $P_{ij}$ is exactly the condition that the image of $f/s$ in $\KK[t_{ij}]$ is divisible by $t_{ij}^{\alpha_{ij}}$. 
The claim now follows.
\end{proof}

For $i=1,\ldots,m$ we have the gap functions
$\lambda_i=\lambda_{R_i}:\ZZ_{\geq 0}^{r_i}\to \ZZ_{\geq 0}$ of $R_i$ in $S_i$:
\[
	\lambda_i(\alpha)=\dim S_i/\left(R_i+\langle t_{i1}^{\alpha_1},\ldots,t_{ir_i}^{\alpha_{r_i}}\rangle\right).
\]
Likewise, we have the gap function 
\[
	\lambda'=\lambda_{R'}:\ZZ_{\geq 0}^{r_1}\times\ldots\times\ZZ_{\geq 0}^{r_m}\to \ZZ_{\geq 0}
\]
of $R'=(1/s)M$ in $S$:
\[
	\lambda'(\alpha)=\dim \left(\prod_{i=1}^m S_i\right)/\left(\frac{1}{s}M+\langle t_{ij}^{\alpha_{r_{ij}}}\rangle\right).
\]
Let $R$ be the subring of $S$ generated by $R'$; then we also have the gap function 
\[\lambda_R:\ZZ_{\geq 0}^{r_1}\times\ldots\times\ZZ_{\geq 0}^{r_m}\to\ZZ
\]
	 of $R$ in $S$.
We remark that the functions $\lambda_i$ are standard gap functions.

These gap functions connect to $\mu$ by the following:
\begin{lemma}\label{lemma:lambda}
	For all $\alpha=(\alpha^1,\ldots,\alpha^m)\in \ZZ_{\geq 0}^{r_1}\times\ldots\times\ZZ_{\geq 0}^{r_m}$,
\[
	\sum_{i=1}^m\lambda_i(\alpha^i)\leq\lambda_R(\alpha)\leq \lambda'(\alpha)= |\alpha|-\mu(\alpha).
\]
\end{lemma}
\begin{proof}
	We begin with the leftmost inequality. Let $\alpha^i=(\alpha_{i1},\ldots,\alpha_{ir_i})$.
Then
\begin{align*}
	\sum_{i=1}^m\lambda_i(\alpha^i)
	=\dim \prod_{i=1}^m\left( S_i/\left(R_i+\langle t_{i1}^{\alpha_{i1}},\ldots,t_{ir_i}^{\alpha_{ir_i}}\rangle\right)\right)
\\	\leq \dim \left ( \prod_{i=1}^m S_i \right) /\left(R+\langle t_{ij}^{\alpha_{ij}}\rangle\right)=\lambda_R(\alpha).
	\end{align*}
Indeed, this follows from that fact that the (finite dimensional) vector space
\[
	A:=\prod_{i,j} \KK[t_{ij}]/\langle t_{ij}^{\alpha_{ij}}\rangle
\]
decomposes as the direct sum of
\[
	A_i:=\prod_{j} \KK[t_{ij}]/\langle t_{ij}^{\alpha_{ij}}\rangle
\]
and the image of $R_i$ in $A_i$ contains the projection of the image of $R$ in $A$ to each $A_i$.

The middle inequality is immediate from $R'\subseteq R$.

For the equality on  the right side, we have
\begin{align*}
\lambda'(\alpha)&=\left (\dim \prod_{i=1}^m S_i \right) /\left(\frac{1}{s} M+\langle t_{ij}^{\alpha_{ij}}\rangle\right)\\
&=\dim\left( \prod_{i=1}^m S_i \right) /\langle t_{ij}^{\alpha_{ij}}\rangle
	-
	\dim \left(\frac{1}{s}M\right)/\left(\frac{1}{s}M\cap \langle t_{ij}^{\alpha_{ij}}\rangle\right)\\
	&=|\alpha|-\mu(\alpha).
\end{align*}
\end{proof}

\begin{cor}\label{cor:l1}
	Consider $\alpha\in\ZZ_{\geq 0}^{r_1}\times\ldots\times\ZZ_{\geq 0}^{r_m}$. If $|\alpha|\leq d+1$, we have
\begin{equation*}
	\sum_{i=1}^m \lambda_i(\alpha^i)\leq \lambda_R(\alpha)\leq \lambda'(\alpha)\leq \dim(L\cap \F^\alpha)+\rho_g.
\end{equation*}
Furthermore, if $|\alpha|\leq d+1-2\rho_g$, we have
\begin{equation*}
	\sum_{i=1}^m \lambda_i(\alpha^i)\leq \lambda_R(\alpha)\leq \lambda'(\alpha)=\dim(L\cap \F^\alpha).
\end{equation*}
\end{cor}
\begin{proof}
By Lemmas~\ref{lemma:mu} and~\ref{lemma:lambda}, we obtain
\[
\lambda'(\alpha)= \dim(L\cap \F^\alpha)+(|\alpha|-\dim \F^{\alpha}).
\]
The claims now follow from Lemma~\ref{lemma:bound}.
\end{proof}

We now relate our functions $\lambda_i$ back to the singularity degrees $\delta(Q_i)$:
\begin{lemma}\label{lemma:delta}
	For any $\alpha\in \ZZ_{\geq 0}^{r_i}$,
$\lambda_i(\alpha)\leq \delta(Q_i)$.
On the other hand, there exists an $\alpha\in \ZZ_{\geq 0}^{r_i}$ for which $\lambda_i(\alpha)=\delta(Q_i)$. 
\end{lemma}
\begin{proof}
	The first claim is immediate. The second claim follows from
	Lemma~\ref{lemma:d2} in \S\ref{sec:convergence}.
\end{proof}

\section{Bounding the Genus}\label{sec:bound}
\begin{thm}\label{thm:b1}
Let $X$ be a smooth projective curve of geometric genus $\rho_g$, $\mcL$ a degree $d$ line bundle, and $\phi:X\to C$ a birational morphism induced by a basepoint free linear system $M \subset H^0(X,\mcL)$. Let $L=M^\perp\subset H^0(X,\mcL)^*$.
Assume that $2\dim L< d-2\rho_g$. 
Then \[\rho_a-\rho_g\leq \dim L,\]
where $\rho_a$ is the arithmetic genus of $C$.
\end{thm}
\begin{proof}
	We consider the gap function $\lambda_R$ from \S\ref{sec:rsing}. We will apply Lemma~\ref{lemma:key} to this function using $\gamma=\dim L$. By Corollary~\ref{cor:l1}, $\lambda_R(\alpha)\leq \gamma$ for all $\alpha\in\ZZ_{\geq 0}^{r_1+\ldots+r_m}$ satisfying 
	\[
	|\alpha|\leq 2\gamma+2=2\dim L+2\leq d+1-2\rho_g.
\]
Thus, $\lambda_R(\alpha)\leq \dim L$ for all $\alpha$. But then 
\[
\rho_a-\rho_g=\sum_i\delta(\lambda_i)\leq \dim L,
\]
where the equality on the left is by Lemma~\ref{lemma:delta}, and the inequality on the right is by Corollary~\ref{cor:l1}.
\end{proof}

We may use this to obtain the following.

\begin{cor}\label{thm:1new}
	Fix $d\geq 1$ and $n\geq 3$.  Assume that $d<2n$.
	Then for any non-degenerate degree $d$ curve $C\subset \PP^n$ with arithmetic genus $\rho_a$, we have
 $\rho_a\leq d-n$. 
\end{cor}
\begin{proof}
Let $\phi:X\to C$ be the normalization of $C$. The morphism $\phi$ is given by a linear system $M\subset H^0(X,\mcL)$ of dimension $n+1$, where $\mcL=\phi^*(\CO_{\PP^n}(1))$.
We let $L=M^\perp \subset H^0(\mcL)^*$ as in Theorem \ref{thm:b1}.

Suppose $h^1(\mcL)> 0$. Then $\mcL$ is special \cite[Example IV.1.3.4]{hartshorne} and Clifford's Theorem \cite[Theorem IV.5.4]{hartshorne} yields $h^0(\mcL)\leq d/2 + 1$, where $h^0(\mcL) = (n+1)+\dim L$. 
Thus, one obtains
\[
n+1+\dim L \leq d/2 + 1< n+1,
\]
which is a contradiction. Therefore $h^1(\mcL) = 0$.

Thus, by Riemann-Roch $n+1+\dim L = d+1-\rho_g$.
This implies
\[
2\dim L = 2d -2\rho_g - 2n = d - 2\rho_g + (d-2n),
\]
and hence $2\dim L < d - 2\rho_g$. Theorem \ref{thm:b1} gives $\rho_a - \rho_g \leq \dim L = d - n-\rho_g$ and the statement follows. 
\end{proof}

\begin{ex}\label{ex:1}
	The upper bound on the arithmetic genus from Corollary~\ref{thm:1new} is sharp. Indeed, for $d\leq 2n-1$, consider the rational curve
	\[
		C=\proj\KK[y^d,x^{d-n+1}y^{n-1},x^{d-n+2}y^{n-2},x^{d-n+3}y^{n-3},\ldots,x^d]\subset\PP^n.
	\]
This has degree $d$ and arithmetic genus $d-n$.
Indeed, this curve is smooth except for at the point $(1:0:\cdots:0)$, where the valuation semigroup contains all of $\NN$ except for $1,2,\ldots,d-n$.
\end{ex}
\begin{ex}\label{ex:2}
The hypothesis $d<2n$ is necessary in order to obtain that the arithmetic genus is at most $d-n$. Indeed, for $d=2n$, consider the rational curve 
	\[
		C=\proj\KK[y^d,x^{n+1}y^{n-1},x^{n+2}y^{n-2},x^{n+3}y^{n-3},\ldots,x^d]\subset\PP^n.
	\]
This curve is a projection of $X_d$ away from a projective linear space of dimension $n-1$, it has degree $d$ and arithmetic genus $n+1=d-n+1$.
Indeed, this curve is smooth except for at the point $(1:0:\cdots:0)$, where the valuation semigroup contains all of $\NN$ except for $1,2,\ldots,n$ and $2n+1$.
\end{ex}

\begin{rem}[{\bf Castelnuovo-Mumford regularity}]
Let $C$ be a non-degenerate degree $d$ curve in $\PP^n$ of arithmetic genus $\rho_a$ and geometric genus $\rho_g$. 
Then \cite[Theorem 1]{Noma} implies that
for any natural number $\ell$ satisfying $\ell<\rho_a$ and $n\geq \ell+2$, $C$ is $(d-n-\ell+2)$-regular.

If we assume that $d<2n$, we may apply Theorem~\ref{thm:1} to obtain that $\ell=\rho_a-1$ satisfies the hypothesis of loc.~cit. We conclude that
$C$ is $\left(d-n-\rho_a+3\right)$-regular.
By well-known properties of Castelnuovo-Mumford
regularity, this implies that $h^0(\mathcal O_C(t)) = h_C(t)$ for all $t\geq (d-n)-\rho_a+2$, where $h_C(t)$ is the Hilbert polynomial of $C$.

In particular, for curves with maximal arithmetic genus ($\rho_a=d-n$) we obtain that $C$ is $3$-regular.
\end{rem}

\section{Classifying Gap Functions}\label{sec:sing} 
\subsection{Classifying Standard Gap Functions for \texorpdfstring{$\delta\leq 3$}{delta at most 3}}\label{sec:classify}
In this section, we will classify all standard gap functions $\lambda_R$ for subalgebras $R\subset S$ with degree $\delta(\lambda)\leq 3$.
As described at the end of \S\ref{sec:prelim}, it suffices to describe the restriction of such $\lambda_R$ to $\NN^r$, which is what we will do in the following.

\begin{rem}
Any time we bound $\delta(\lambda)$, there are only finitely many possible standard gap functions $\lambda$ coming from subalgebras of $S$. Moreover, $\lambda$ is determined by its values on those $\alpha\in \NN^r$ satisfying $|\alpha|\leq 2\delta$. 

To see this, consider the natural partial order on $\NN^r$. For any $\alpha\in \NN^r$ satisfying $|\alpha|> 2\delta$, $\lambda(\alpha)=\max_{\alpha'}\lambda(\alpha')$, where the maximum is taken over all $\alpha'$ smaller than $\alpha$. Indeed, set 
$\gamma=\max_{\alpha'}\lambda(\alpha')$; certainly $\lambda(\alpha)\geq \gamma$, so if $\gamma=\delta$ we are done.
If $\gamma<\delta$, then $|\alpha|>2\gamma+2$, and we may apply Lemma~\ref{lemma:help} to conclude that $\lambda(\alpha)\leq \gamma$ as desired.
\end{rem}

To describe any such gap function $\lambda$ with $\delta\leq 3$, it suffices to identify those positions where $\lambda(\alpha)$ is larger than $\lambda(\alpha-e_i)$ for all $i$ with $\alpha-e_i\in \NN^r$. We highlight those positions in blue in the Tables~\ref{table:d1}, \ref{table:d2}, and~\ref{table:d3}. Note that the tables are only listing the values of $\lambda$ on $\NN^r$.

\begin{table}
	\[	\begin{array}{l@{\qquad} l}
		\textbf{1.1}:\quad		\begin{array} {|c c c c c}
			\hl{0}& \hl{1} & 1 & \ldots 
		\end{array}
	\\
	\\
	\\
		\textbf{1.2}:\quad	\begin{array}{|c c c  }
		\vdots&\vdots\\
		1&1&\ldots\\
		\hl{1}&1&\ldots\\
		\hline
	\end{array}&
\end{array}
\]
	\caption{Functions $\lambda$ with  $\delta=1$}\label{table:d1}
\end{table}

\begin{table}
	\[	\begin{array}{l@{\qquad} l}
		\textbf{2.1.a}:\quad		\begin{array} {|c c c c c}
			\hl{0}& \hl{1} & \hl{2} & \ldots 
		\end{array}
		&
\textbf{2.1.b}:\quad		\begin{array} {|c c c c c c c}
			\hl{0}& \hl{1} & 1&\hl{2} &  \ldots 
		\end{array}
	\\
	\\
	\\
	\textbf{2.2.a}:\quad	\begin{array}{|c c c }
		\vdots&\vdots\\
		1&\hl{2}&\ldots\\
		\hl{1}&1&\ldots\\
		\hline
	\end{array}&
	\textbf{2.2.b}:\quad	\begin{array}{|c c c }
		\vdots&\vdots\\
		 1&2&\ldots\\
		\hl{1}&\hl{2}&\ldots\\
		\hline
	\end{array}
\end{array}
\]
	\begin{align*}	
&\textbf{2.3}:\quad \lambda(1,1,1)=\hl 2
\end{align*}
	\caption{Functions $\lambda$ with  $\delta=2$}\label{table:d2}
\end{table}

\begin{table}
	\[	\begin{array}{l@{\qquad} l}
		\textbf{3.1.a}:\quad		\begin{array} {|c c c c c}
			\hl{0}& \hl{1} & \hl{2} & \hl{3}  & \ldots 
		\end{array}
		&
		\textbf{3.1.b}:\quad		\begin{array} {|c c c c c c}
			\hl{0}& \hl{1} & \hl{2} & 2&\hl{3}  & \ldots 
		\end{array}\\
		\\
		\\
		\textbf{3.1.c}:\quad		\begin{array} {|c c c c c c c}
			\hl{0}& \hl{1} & \hl{2} & 2&2 & \hl{3} & \ldots 
		\end{array}
&
\textbf{3.1.d}:\quad		\begin{array} {|c c c c c c c}
			\hl{0}& \hl{1} & 1&\hl{2} & 2& \hl{3} & \ldots 
		\end{array}
	\\
	\\
	\\
\textbf{3.2.a}:\quad	\begin{array}{|c c c c c}
		\vdots&\vdots&\vdots\\
		1&2&3&\ldots\\
		1& 2&3&\ldots\\
		\hl{1}&\hl{2}&\hl{3}&\ldots\\
		\hline
	\end{array}&
\textbf{3.2.b}:\quad	\begin{array}{|c c c c c}
		\vdots&\vdots&\vdots\\
		1&2&2&3&\ldots\\
		1& 2&2&3&\ldots\\
		\hl{1}&\hl{2}&2&\hl{3}&\ldots\\
		\hline
	\end{array}\\ \\ \\
\textbf{3.2.c}:\quad	\begin{array}{|c c c c c}
		\vdots&\vdots&\vdots\\
		2&3&3&\ldots\\
		\hl{2}& \hl{3}&3&\ldots\\
		\hl{1}&\hl{2}&2&\ldots\\
		\hline
	\end{array} &
\textbf{3.2.d}:\quad	\begin{array}{|c c c c c}
		\vdots&\vdots&\vdots\\
		1&2&3&\ldots\\
		1& 2&\hl{3}&\ldots\\
		\hl{1}&\hl{2}&2&\ldots\\
		\hline
	\end{array}
\\ \\ \\
\textbf{3.2.e}:\quad	\begin{array}{|c c c c c}
		\vdots&\vdots&\vdots\\
		1&2&2&3&\ldots\\
		1& 2&2&\hl{3}&\ldots\\
		\hl{1}&\hl{2}&2&2&\ldots\\
		\hline
	\end{array}
	&\textbf{3.2.f}:\quad	\begin{array}{|c c c c}
		\vdots&\vdots&\vdots\\
		1&2&\hl{3}&\ldots\\
		1&\hl{2}&2&\ldots\\
		\hl{1}&1&1&\ldots\\
		\hline
	\end{array}
	\\
	\\
\end{array}
\]
	\begin{align*}	
	&\textbf{3.3.a}:\qquad &\lambda(1,1,1)=\hl{2}, \quad\lambda(1,1,2)=\hl{3}\\
		&\textbf{3.3.b}:\qquad &\lambda(1,1,1)=\hl{2}, \quad\lambda(1,2,2)=\hl{3}\\
		&\textbf{3.3.c}:\qquad &\lambda(1,1,1)=\hl{2},\quad\lambda(2,2,2)=\hl{3}\\
&\textbf{3.4}:\quad &\lambda(1,1,1,1)=\hl 3
\end{align*}
	\caption{Functions $\lambda$ with  $\delta=3$}\label{table:d3}
\end{table}

Before we proceed with classifying the possible functions $\lambda$ with $\delta\leq 3$, we commence with a useful general observation:

\begin{rem}[{\bf Upward propagation}]\label{upward}
By Remark~\ref{rem:sigma}, if $\lambda(i,j)<\lambda(i+1,j)$, then $\lambda(i,k)<\lambda(i+1,k)$ for all $k>j$. Indeed, were $\lambda(i,k)=\lambda(i+1,k)$, then $\lambda(i,k)[1]\in \Sigma$, which in turn would have implied that $\lambda(i,j)=\lambda(i+1,j)$. We call such a behavior \emph{upward propagation}.
\end{rem}

\begin{prop}[{See also \cite[\S2.1]{ethan1}}]\label{prop:classifylambda}
The possible standard gap functions $\lambda$ for subalgebras $R$ of $S$ with $\delta(\lambda)\leq 3$ are, up to permutation of the $r$ coordinates, those listed in 
Tables~\ref{table:d1}, \ref{table:d2}, and~\ref{table:d3} .
\end{prop}

\begin{small}
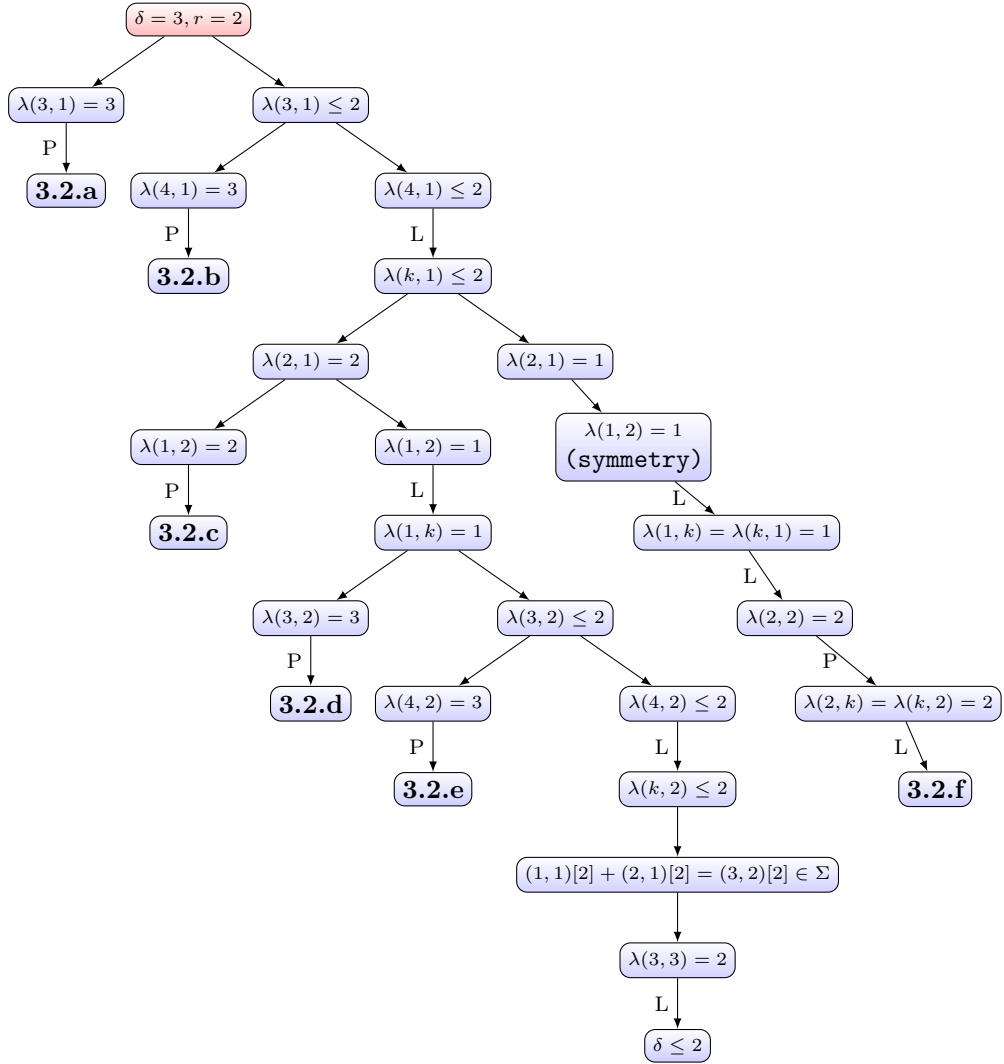
\begin{figure}
\begin{tikzpicture}
  [
    grow                    = down,
    sibling distance        = 10em,
    level distance          = 3.5em,
    edge from parent/.style = {draw, -latex},
    every node/.style       = {font=\footnotesize},
  ]
  \node [root] {\scriptsize{$\delta=3,r=2$}}
  child { node [env] {\scriptsize $\lambda(3,1)=3$}
  	child { node [env] {\bf 3.2.a}
	edge from parent node [left] {P}}}
	child { node [env] {\scriptsize $\lambda(3,1)\leq 2$}
	child { node [env] {\scriptsize $\lambda(4,1)=3$}
		child {node [env] {\bf 3.2.b}
		edge from parent node [left] {P}}}
	child { node [env] {\scriptsize $\lambda(4,1)\leq 2$}
	child { node [env] {\scriptsize $\lambda(k,1)\leq 2$}
		child { node [env] {\scriptsize $\lambda(2,1)=2$}
			child { node [env] {\scriptsize $\lambda(1,2)=2$}
			child { node [env] {\bf 3.2.c}
			edge from parent node [left] {P}
		}
		}
			child { node [env] {\scriptsize $\lambda(1,2)=1$}
		child { node [env] {\scriptsize $\lambda(1,k)=1$}
	child { node [env] {\scriptsize $\lambda(3,2)=3$}
			child { node [env] {\bf 3.2.d}
			edge from parent node [left] {P}}
}
	child { node [env] {\scriptsize $\lambda(3,2)\leq 2$}
	child { node [env] {\scriptsize $\lambda(4,2)=3$}
	child { node [env] {\bf 3.2.e}
			edge from parent node [left] {P}}
}
	child { node [env] {\scriptsize $\lambda(4,2)\leq 2$}
	child { node [env] {\scriptsize $\lambda(k,2)\leq 2$}
	child { node [env] {\scriptsize $(1,1)[2]+(2,1)[2]=(3,2)[2]\in\Sigma$}
	child { node [env] {\scriptsize $\lambda(3,3)=2$}
	child { node [env] {\scriptsize $\delta\leq 2$}
	edge from parent node [left] {L}}}
}
	edge from parent node [left] {L}}}
}
	edge from parent node [left] {L}
		}
	}}
	child { node [env] {\scriptsize $\lambda(2,1)=1$}
	child [right] { node [env] {\scriptsize $\lambda(1,2)=1$\\(symmetry)}
	child [right] { node [env] {\scriptsize $\lambda(1,k)=\lambda(k,1)=1$}
	child [right] { node [env] {\scriptsize $\lambda(2,2)=2$}
	child [right] { node [env] {\scriptsize $\lambda(2,k)=\lambda(k,2)=2$}
	child [right] { node [env] {\bf 3.2.f}
	edge from parent node [left] {L}
}
	edge from parent node [left] {P}
}
	edge from parent node [left] {L}
}
	edge from parent node [left] {L}
}}
	}
	edge from parent node [left] {L}
	}}
};
\end{tikzpicture}
\caption{The flow diagram representing the proof of Proposition~\ref{prop:classifylambda}}\label{fig:proof}
\end{figure}
\end{small}

\begin{proof}
	In Tables~\ref{table:s1}, \ref{table:s2}, \ref{table:s31}, \ref{table:s32}, \ref{table:s33}, we explicitly give subalgebras $R\hookrightarrow \prod \KK[t_i]$ which have the prescribed gap functions $\lambda$. See the proof of Proposition~\ref{prop:sing} for details on the tables.
Here we show that no other gap functions $\lambda$ are possible. 

	This is achieved via repeated use of Remark~\ref{rem:sigma} and Lemma~\ref{lemma:key}. We do this explicitly for the most involved case, which is when $\delta=3$ and $r=2$. The other cases are resolved using similar techniques.

	We will differentiate between a number of different cases. The proof is encoded in Figure~\ref{fig:proof}. The label ``L'' means we are applying Lemma~\ref{lemma:help}, whereas label ``P'' means we are using upward propagation discussed in Remark~\ref{upward}. 

We begin by assuming that $\lambda(3,1)=3$. Then we are in case {\bf 3.2.a}. Indeed, we are forced to have $\lambda(2,1)=2$ and $\lambda(1,1)=1$. Therefore, the structure of the rest of the values of $\lambda$ follows immediately by upward propagation.

Henceforth, we will assume that $\lambda(3,1)\leq 2$. Assume next that $\lambda(4,1)=3$. 
Then $\lambda(3,1)=\lambda(2,1)=2$. Indeed, $\lambda(3,1)\geq \lambda(4,1)-1$. Moreover, we cannot have $\lambda(2,1) = 1$, as
otherwise $(1,1)[1]\in \Sigma$, which would imply $\lambda(3,1) = \lambda(2,1)=1$. Upward propagation completely determines the rest of the values of $\lambda$; then we are in case {\bf 3.2.b}.

So we may assume $\lambda(4,1)\leq 2$. Then Lemma~\ref{lemma:help} applied inductively to $(k,1)$ for $k\geq 5$ implies then that $\lambda(k,1)\leq 2$ for all $k\geq 5$. 

Now assume that $\lambda(2,1)=2$. If $\lambda(1,2)=2$, then upward propagation implies we are in case {\bf 3.2.c}. 
By symmetry, we may thus assume that $\lambda(1,2)=1$ in all other cases. Lemma~\ref{lemma:key} implies that $\lambda(1,k)=1$ for all $k$. 

Maintaining the assumption that $\lambda(2,1)=2$, we next suppose $\lambda(3,2)=3$. Upward propagation implies that we are in case {\bf 3.2.d}. Hence, we may assume $\lambda(3,2)\leq 2$. If $\lambda(4,2)=3$, upward propagation implies we are in the case {\bf 3.2.e}. If instead, $\lambda(4,2)\leq 2$, then Lemma~\ref{lemma:help} applied repeatedly implies that $\lambda(k,2)\leq 2$ for all $k$. Furthermore, since $\lambda(1,1)=\lambda(1,2)$ and $\lambda(2,1)=\lambda(2,2)$, we see that $(1,1)[2],(2,1)[2]\in \Sigma$. Then $(3,2)[2]\in \Sigma$, and so $\lambda(3,3)=2$. Now applying Lemma~\ref{lemma:help} repeatedly would imply that $\lambda(i,j)\leq 2$ for all $(i,j)\in\NN^2$, which is impossible.

So this leaves us with the case that $\lambda(2,1)$ (and by symmetry $\lambda(1,2)$) are both equal to one. Repeated application of Lemma~\ref{lemma:help} gives us that $\lambda(i,j)=1$ when $i$ or $j$ is one. Again, Lemma~\ref{lemma:help} gives $\lambda(2,2)=2$. Now, upward propagation (and its horizontal variant) give that $\lambda(2,k)=\lambda(k,2)=2$, for all $k\geq 3$. If $\lambda(3,3)\neq 3$, then Lemma~\ref{lemma:help} would imply $\lambda(i,j)\leq 2$ for all $(i,j)\in \NN^2$, which is impossible. Finally, we deduce $\lambda(3,3)=3$, and we are in case {\bf 3.2.f}.
\end{proof}

\subsection{Singularities Corresponding to Standard Gap Functions}\label{ssec:sing}
Having classified all standard gap functions $\lambda$ for subalgebras of $S$ with $\delta(\lambda)\leq 3$ in \S\ref{sec:classify}, we now wish to see what each such function $\lambda$ corresponds to geometrically. See also \cite[\S 2.2]{ethan1} for a similar discussion.
\begin{prop}\label{prop:sing}
	Let $\lambda$ be one of the standard gap functions classified in Proposition~\ref{prop:classifylambda}, with $\lambda=\lambda_R$ for some subalgebra $R$ of $S$. Then the completion of the subalgebra $R$ is isomorphic to a quotient of a power series ring with generators and relations as listed in Tables~\ref{table:s1}, \ref{table:s2}, \ref{table:s31}, \ref{table:s32}, \ref{table:s33}. 
\end{prop}

\begin{table}
\small{	\[
	\begin{array}{|l| l| l | l |l|}
	\hline
		\lambda & \textrm{Elements in $\Sigma$} & \textrm{Generators} & \textrm{Relations} & \textrm{Description}\\
\hline
&&&&\\
\textrm{\bf 1.1} & 
\begin{array}{l}
2,3
\end{array}
&
\begin{array}{l}
x=t_1^2,y=t_1^3\\
\end{array}& 
\begin{array}{l}
x^3-y^2
\end{array}
& 
\begin{array}{l}
	\textrm{Cusp}\\
\end{array}\\
& & & &  \\
\hline
&&&&\\
\textrm{\bf 1.2} & 
\begin{array}{l}
(\infty,1),
(1,\infty)\\
\end{array}
&
\begin{array}{l}
x=t_1\\
y=t_2
\end{array}& 
\begin{array}{l}
xy
\end{array}
& 
\begin{array}{l}
	\textrm{Node}\\
\end{array}\\
& & & &  \\
\hline
\end{array}
\]}
\caption{Singularities with $\delta=1$}\label{table:s1}
\end{table}
\begin{table}
\small{	\[
	\begin{array}{|l| l| l | l |l|}
	\hline
		\lambda & \textrm{Elements in $\Sigma$} & \textrm{Generators} & \textrm{Relations} & \textrm{Description}\\
\hline
&&&&\\
\textrm{\bf 2.1.a} & 
\begin{array}{l}
3,4,5
\end{array}
&
\begin{array}{l}
x_1=t_1^3,x_2=t_1^4\\
x_3=t_1^5
\end{array}& 
\begin{array}{l}
x_1x_3-x_2^2\\
x_1^3-x_2x_3\\
x_1^2x_2-x_3^2
\end{array}
& 
\begin{array}{l}
	\textrm{$(3,4,5)$-cusp}\\
\end{array}\\
& & & &  \\
\hline
&&&&\\
\textrm{\bf 2.1.b} & 
\begin{array}{l}
2,4,5
\end{array}
&
\begin{array}{l}
x=t_1^2,y=t_1^5\\
\end{array}& 
\begin{array}{l}
x^5-y^2
\end{array}
& 
\begin{array}{l}
	\textrm{Rhamphoid cusp}\\
\end{array}\\
& & & &  \\
\hline
&&&&\\
\textrm{\bf 2.2.a} & 
\begin{array}{l}
(\infty,2),
(\infty,3)\\
(2,\infty),
(3,\infty)\\
(1,1)\\
\end{array}
&
\begin{array}{l}
x=t_1+t_2\\
y=t_2^2
\end{array}& 
\begin{array}{l}
y(x^2-y)
\end{array}
& 
\begin{array}{l}
	\textrm{Tacnode}\\
\end{array}\\
& & & &  \\
\hline
&&&&\\
\textrm{\bf 2.2.b} & 
\begin{array}{l}
(2,\infty)\\
(3,\infty)\\
(\infty,1)\\
\end{array}
&
\begin{array}{l}
x=t_2\\
y=t_1^2\\
z=t_1^3\\
\end{array}& 
\begin{array}{l}
xy,xz\\
y^3-z^2
\end{array}
& 
\begin{array}{l}
	\textrm{Cusp with}\\
	\textrm{smooth branch}\\
\end{array}\\
& & & &  \\
\hline
&&&&\\
\textrm{\bf 2.3} & 
\begin{array}{l}
(1,\infty,\infty)\\
(\infty,1,\infty)\\
(\infty,\infty,1)\\
\end{array}
&
\begin{array}{l}
x=t_1\\
y=t_2\\
z=t_3
\end{array}& 
\begin{array}{l}
xy,xz\\
yz
\end{array}
& 
\begin{array}{l}
	\textrm{Ordinary}\\
	\textrm{triple point}\\
\end{array}\\
& & & &  \\
\hline
\end{array}
\]}
\caption{Singularities with $\delta=2$}\label{table:s2}
\end{table}

\begin{table}
\small{	\[
	\begin{array}{|l| l| l | l |l|}
	\hline
		\lambda & \textrm{Elements in $\Sigma$} & \textrm{Generators} & \textrm{Relations} & \textrm{Description}\\
\hline
&&&&\\
\textrm{\bf 3.1.a} & 
\begin{array}{l}
4,5,6,7
\end{array}
&
\begin{array}{l}
x_1=t_1^4,x_2=t_1^5\\
x_3=t_1^6,x_4=t_1^7
\end{array}& 
\begin{array}{l}
x_1x_3-x_2^2\\
x_1x_4-x_3^3\\
x_2x_4-x_3^2\\
x_1^2x_3-x_4^2\\
x_1^2x_2-x_3x_4\\
x_1^3-x_2x_4\\
\end{array}
& 
\begin{array}{l}
	\textrm{$(4,5,6,7)$-cusp}\\
\end{array}\\
& & & &  \\
\hline
&&&&\\
\textrm{\bf 3.1.b} & 
\begin{array}{l}
3,5,7
\end{array}
&
\begin{array}{l}
x_1=t_1^3,x_2=t_1^5\\
x_3=t_1^7\\
\end{array}& 
\begin{array}{l}
x_1x_3-x_2^2\\
x_1^3x_2-x_3^2\\
x_2x_3-x_1^4\\
\end{array}
& 
\begin{array}{l}
	\textrm{$(3,5,7)$-cusp}\\
\end{array}\\
& & & &  \\
\hline
&&&&\\
\textrm{\bf 3.1.c} & 
\begin{array}{l}
3,4 
\end{array}
&
\begin{array}{l}
x =t_1^3,y =t_1^4\\
\end{array}
& 
\begin{array}{l}
x^4-y^3
\end{array}
& 
\begin{array}{l}
	\textrm{$(3,4)$-cusp}\\
\end{array}\\
& & & &  \\
\hline
&&&&\\
\textrm{\bf 3.1.d} & 
\begin{array}{l}
2,7
\end{array}
&
\begin{array}{l}
x=t_1^2,y=t_1^7\\
\end{array}& 
\begin{array}{l}
x^7-y^2\\
\end{array}
& 
\begin{array}{l}
	\textrm{$(2,7)$-cusp}\\
\end{array}\\
& & & &  \\
\hline
\end{array}
\]}
\caption{Singularities with $\delta=3,r=1$}\label{table:s31}
\end{table}

\begin{table}
\small{	\[
	\begin{array}{|l| l| l | l |l|}
	\hline
		\lambda & \textrm{Elements in $\Sigma$} & \textrm{Generators} & \textrm{Relations} & \textrm{Description}\\
\hline
&&&&\\
\textrm{\bf 3.2.a} & 
\begin{array}{l}
(3,\infty), (4,\infty)\\ (5,\infty), (\infty,1)
\end{array}
&
\begin{array}{l}
x_1=t_1^3,x_2=t_1^4\\
x_3=t_1^5,y=t_2
\end{array}& 
\begin{array}{l}
x_1y,x_2y,x_3y\\
x_1x_3-x_2^2\\
x_1^3-x_2x_3\\
x_1^2x_2-x_3^2
\end{array}
& 
\begin{array}{l}
	\textrm{$(3,4,5)$-cusp with}\\
\textrm{smooth branch}
\end{array}\\
& & & &  \\
\hline
&&&&\\
\textrm{\bf 3.2.b} & 
\begin{array}{l}
(2,\infty), (5,\infty)\\ (\infty,1)
\end{array}
&
\begin{array}{l}
x_1=t_1^2,x_2=t_1^5\\y=t_2
\end{array}& 
\begin{array}{l}
x_1y,x_2y\\
x_1^5-x_2^2
\end{array}
& 
\begin{array}{l}
	\textrm{Rhamphoid cusp}\\
	\textrm{with smooth branch}
\end{array}\\
&&&&\\
\hline
&&&&\\
\textrm{\bf 3.2.c} & 
\begin{array}{l}
(2,\infty), (3,\infty)\\ (\infty,2),(\infty,3)
\end{array}
&
\begin{array}{l}
x_1=t_1^2,x_2=t_1^3\\y_1=t_2^2,y_2=t_2^3
\end{array}& 
\begin{array}{l}
x_1y_1,x_1y_2\\
x_2y_1,x_2y_2\\
x_2^3-x_2^2\\
y_2^3-y_2^2
\end{array}
& 
\begin{array}{l}
	\textrm{Two independent}\\
	\textrm{cusps}
\end{array}\\
&&&&\\
\hline
&&&&\\
\textrm{\bf 3.2.d} & 
\begin{array}{l}
(3,\infty), (4,\infty)\\ 
(5,\infty), (\infty,2)\\
(\infty,3), (2,1)
\end{array}
&
\begin{array}{l}
x=t_1^2+t_2\\y=t_1^3\\z=t_2^2
\end{array}& 
\begin{array}{l}
yz\\ z(x^2-z) \\ x^3-y^2-xz
\end{array}
& 
\begin{array}{l}
	\textrm{Cusp with}\\
	\textrm{collinear}\\
	\textrm{smooth branch}\\
\end{array}\\
&&&&\\
\hline
&&&&\\
\textrm{\bf 3.2.e} & 
\begin{array}{l}
(2,\infty), (5,\infty)\\ 
(\infty,2),(\infty,3)\\
(3,1)
\end{array}
&
\begin{array}{l}
x=t_1^3+t_2\\y=t_1^2
\end{array}& 
\begin{array}{l}
y(x^2-y^3)
\end{array}
& 
\begin{array}{l}
	\textrm{Cusp with}\\
	\textrm{coplanar}\\
	\textrm{smooth branch}\\
\end{array}\\
&&&&\\
\hline
&&&&\\
\textrm{\bf 3.2.f} & 
\begin{array}{l}
(3,\infty), (4,\infty)\\ 
(5,\infty), (\infty,3)\\ 
(\infty,4),(\infty,5)\\
(1,1)
\end{array}
&
\begin{array}{l}
x=t_1+t_2\\y=t_1^3
\end{array}& 
\begin{array}{l}
y(x^3-y)
\end{array}
& 
\begin{array}{l}
	\textrm{Node with}\\
	\textrm{third order}\\
	\textrm{contact}\\
\end{array}\\
&&&&\\
\hline
\end{array}
\]}
\caption{Singularities with $\delta=3,r=2$}\label{table:s32}
\end{table}

\begin{table}
\small{	\[
	\begin{array}{|l| l| l | l |l|}
	\hline
		\lambda & \textrm{Elements in $\Sigma$} & \textrm{Generators} & \textrm{Relations} & \textrm{Description}\\
\hline
&&&&\\
\textrm{\bf 3.3.a} & 
\begin{array}{l}
(1,\infty,\infty)\\
(\infty,1,\infty)\\
(\infty,\infty,2)\\
(\infty,\infty,3)\\
\end{array}
&
\begin{array}{l}
x_1=t_1,x_2=t_2\\
y=t_3^2,z=t_3^3
\end{array}& 
\begin{array}{l}
x_1x_2,x_1y\\
x_1z,x_2y\\
x_2z\\
y^3-z^2
\end{array}
& 
\begin{array}{l}
	\textrm{Cusp with 2}\\
\textrm{smooth branches}
\end{array}\\
& & & &  \\
\hline
&&&&\\
\textrm{\bf 3.3.b} & 
\begin{array}{l}
(1,\infty,\infty)\\
(\infty,2,\infty)\\
(\infty,3,\infty)\\
(\infty,\infty,2)\\
(\infty,\infty,3)\\
(1,1,1)\\
\end{array}
&
\begin{array}{l}
x=t_1,y=t_2^2\\
z=t_2+t_3
\end{array}& 
\begin{array}{l}
xy,xz\\
y(z^2-y)
\end{array}
& 
\begin{array}{l}
	\textrm{Tacnode with}\\
	\textrm{extra smooth}\\
\textrm{branch}
\end{array}\\
& & & &  \\
\hline
&&&&\\
\textrm{\bf 3.3.c} & 
\begin{array}{l}
(2,\infty,\infty)\\
(3,\infty,\infty)\\
(1,1,2)\ \textrm{and}\\
\textrm{all permutations}
\end{array}
&
\begin{array}{l}
x=t_1+t_2\\
y=t_1+t_3
\end{array}& 
\begin{array}{l}
xy(x-y)
\end{array}
& 
\begin{array}{l}
	\textrm{Planar}\\
	\textrm{triple point}\\
\end{array}\\
& & & &  \\
\hline
&&&&\\
\textrm{\bf 3.4} & 
\begin{array}{l}
(1,\infty,\infty,\infty)\\
(\infty,1,\infty,\infty)\\
(\infty,\infty,1,\infty)\\
(\infty,\infty,\infty,1)\\
\end{array}
&
\begin{array}{l}
x=t_1\\
y=t_2\\
z=t_3\\
w=t_4
\end{array}& 
\begin{array}{l}
xy,xz,xw\\
yz,yw,zw
\end{array}
& 
\begin{array}{l}
	\textrm{Ordinary}\\
	\textrm{quadruple point}\\
\end{array}\\
& & & &  \\
\hline
\end{array}
\]}
\caption{Singularities with $\delta=3,r\geq 3$}\label{table:s33}
\end{table}

\begin{proof}
We outline our general argument here; the proof follows by inspection of the relevant tables.
By completing when necessary, we will always assume that our gap function $\lambda$ is obtained from a complete subalgebra $R$ of $S$.
\begin{enumerate}

\item[{\bf Step 1.}] Given $\lambda$, 
	we determine specific elements of $\Sigma$ and list them in the table. These elements are constructed using two techniques; the first is Lemma~\ref{lemma:convergence}. The second is the final observation of Remark~\ref{rem:sigma}.
\item[{\bf Step 2.}] For some specific $\alpha\in\Sigma$, we produce functions $f_\alpha\in R$ with $\val(f_\alpha)=\alpha$ and list them as generators. To guarantee that we can choose the $f_\alpha$ in the form we specify, we use ring automorphisms of $S$, paired with the existence of certain elements in $R$ of higher valuation.

\item[{\bf Step 3.}] Let $R'$ be the completion of the subalgebra inside $S$ generated by the $f_{\alpha}$. An explicit computation in each case shows that
	$\lambda=\lambda_{R'}$.

\item[{\bf Step 4.}] Compute the relations among the generators $f_{\alpha}$.
This is again a straightforward computation, which e.g.~may be carried out with the help of a computer algebra system.
\end{enumerate} 
\end{proof}

We illustrate the steps of the above proof in two examples. 

\begin{ex}[{\bf 3.1.c}]
The number of branches is $r=1$ and so $S = \KK[[t_1]]$. 
Since $\lambda(3) = \lambda(4) = \lambda(5)$, $3,4\in \Sigma$; likewise, $k\in \Sigma$, for $k\geq 6$. Hence there exists a function $x'\in R$ such that
$\val(x') = 3$. By an automorphism of $S$, we may assume $x' = t_1^3$.

Let $y'\in R$ be an element such that $\val(y') = 4$. So 
$y' = t_1^4+c\cdot t_1^5+\ldots$. Since $k\in \Sigma$ for all $k\geq 6$, we have functions $f_k\in R$ with valuation $\val(f_k) = k$, for all $k\geq 6$. 
We use these functions to get rid of the terms of  $y'$ of order at least $6$. In conclusion, we may assume that 
\[
x' = t_1^3; \\
y'= t_1^4 + c\cdot t_1^5
\]
are both in $R$.

We now consider the automorphism 
\[t_1\mapsto t_1-\frac{c}{4}t_1^2-\frac{c^2}{16}t_1^3.\]
The image of $y'$ is of the form $t_1^4+s$, where $\val(s)\geq 6$.
Likewise the image of $x'+\frac{3c}{4}y'$ is of the form $t_1^3+r$, where $\val(r)\geq 6$. Using the images of the $f_k$, $k\geq 6$, we may thus assume that 
\[
x = t_1^3; \\
y= t_1^4 
\]
are both in $R$.

The subalgebra generated by these two elements has the same function $\lambda$, so its completion must equal the completion of $R$. The unique relation between the generators $x$ and $y$ is $x^4-y^3$, and so the completion of $R$ has a presentation
\[
R \cong \KK[[x, y]]/(x^4-y^3).
\]
\end{ex}

\begin{ex}[{\bf 3.2.b}]
The number of branches is $r=2$ and so $S = \KK[[t_1]]\times\KK[[t_2]]$.  
Lemma~\ref{lemma:convergence} gives that $(2,\infty),(5,\infty),(\infty,1)$ all belong to $\Sigma$.

Let $x_1$ be a function whose valuation is $\val(x_1) = (2,\infty)$. Up to a ring automorphism of $\KK[[t_1]]$, 
we may assume $x_1 = t_1^2$. 
Let $x_2$ be a function whose valuation is $\val(x_2) = (5,\infty)$. Since $(k,\infty)\in \Sigma$ for all $k\geq 6$, we may kill any higher valuation terms of $x_2$ and assume that $x_2=t^5$. 

Finally, let $y$ be a function whose valuation is $\val(y) = (\infty,1)$. Again, up to an automorphism of $\KK[[t_2]]$, we may assume 
$y = t_2$. 
The completion of the algebra generated by $x_1,x_2,y$ agrees with the completion of $R$, and has presentation
\[
R'\cong \KK[[x_1,x_2,y]]/(x_iy, x_1^5-x_2^2).
\]
This is the union of a rhamphoid cusp with a smooth branch, with independent tangent spaces.
\end{ex}

\subsection{Gap Functions for Finite-Dimensional Vector Spaces}
Instead of considering a complete subalgebra $R\hookrightarrow S=\prod_{i=1}^r \KK[[t_i]]$,
we now consider a finite-dimensional $\KK$-vector subspace $R'$ of $S$ such that $\lambda'=\lambda_{R'}$ is a standard gap function.
Letting $R$ be the subring of $S$ generated by $R'$, we ask: what can we say about $\lambda_R$ based on $\lambda'$?
This is answered by the following proposition:
\begin{prop}\label{prop:linear}
Let $R'\subset S$ be a finite dimensional vector space leading to a standard gap function $\lambda'=\lambda_{R'}$. Assume that $R'$ contains $(1,\ldots,1)\in S$.  Let $R$ be the subalgebra of $S$ generated by $R'$, and assume that $\delta(\lambda_R)\leq 3$. Then $\lambda_R$ is one of the standard gap functions $\lambda$ classified in Proposition~\ref{prop:classifylambda} if and only if $\lambda'$ fulfills the conditions in the corresponding row of Table~\ref{table:vs}.
\end{prop}

We remark that  some of the cases  in Table~\ref{table:vs} include
  the condition $\lambda'(1, \dotsc, 1) = r-1$.
In the setting of Proposition~\ref{prop:linear} this seems redundant,
  as the gap function $\lambda'$ is required to be standard.
However, we will state a slightly stronger implication in 
  Lemma~\ref{lem:stronger_for_vs} for which this condition is important.

\begin{table}
	\begin{tabular}{|l |l |}
		\hline
\rule{0pt}{4ex}$\lambda$ & Conditions on $\lambda'$\\	
\hline
\rule{0pt}{4ex}{\bf 1.1} & $\lambda'(2)=\lambda'(4)=1$ \\
\rule{0pt}{4ex}{\bf 1.2} & $\lambda'(1,1)=\lambda'(2,2)=1$ \\ 
\hline
\hline
\rule{0pt}{4ex}{\bf 2.1.a} & $\lambda'(3)=\lambda'(5)=2$ \\
\rule{0pt}{4ex}{\bf 2.1.b} or {\bf 3.1.d}& $\lambda'(2)=\lambda'(3)=1,\lambda'(4)=2$ \\
\rule{0pt}{4ex}{\bf 2.2.a} or {\bf 3.2.f} & \begin{tabular}{@{}l}$\lambda'(1,1)=\lambda'(1,2)=\lambda'(2,1)=1$, $\lambda'(2,2)=2$;
 \end{tabular} \\
\rule{0pt}{4ex}{\bf 2.2.b} & \begin{tabular}{@{}l} $\lambda'(2,1)=\lambda'(4,2)=2$\end{tabular} \\
\rule{0pt}{4ex}{\bf 2.3} & $\lambda'(1,1,1)=\lambda'(2,2,2)=2$ \\
\hline
\hline
\rule{0pt}{4ex}{\bf 3.1.a} & $\lambda'(4)=3$  \\
\rule{0pt}{4ex}{\bf 3.1.b} & $\lambda'(3)=\lambda'(4)=2$, $\lambda'(5)=3$ \\
\rule{0pt}{4ex}{\bf 3.1.c} & $\lambda'(3)=\lambda'(5)=2$, $\lambda'(6)=3$ \\
\hline
\hline
\rule{0pt}{4ex}{\bf 3.2.a} & $\lambda'(3,1)=3$  \\
\rule{0pt}{4ex}{\bf 3.2.b} &  $\lambda'(2,1)=\lambda'(3,2)=2$, $\lambda'(4,1)=3$\\
\rule{0pt}{4ex}{\bf 3.2.c} &  $\lambda'(2,2)=3$\\
\rule{0pt}{4ex}{\bf 3.2.d} &  $\lambda'(2,1)=\lambda'(2,2)=\lambda'(3,1)=2$, $\lambda'(3,2)=3$\\
\rule{0pt}{4ex}{\bf 3.2.e} &  $\lambda'(2,1)=\lambda'(3,2)=\lambda'(4,1)=2$, $\lambda'(4,2)=3$ \\
\hline
\hline
\rule{0pt}{4ex}{\bf 3.3.a} & $\lambda'(1,1,2)=3$  \\
\rule{0pt}{4ex}{\bf 3.3.b} & \begin{tabular}{@{}l} $\lambda'(1,1,1)=\lambda'(1,1,2)=\lambda'(1,2,1)=\lambda'(2,1,1)=2$;\\  $\lambda'(1,2,2)=3$ \end{tabular}\\
\rule{0pt}{4ex}{\bf 3.3.c} & \begin{tabular}{@{}l} $\lambda'(1,1,1)=\lambda'(1,2,2)=\lambda'(2,1,2)=\lambda'(2,2,1)=2$;\\  $\lambda'(2,2,2)=3$ \end{tabular}\\
\rule{0pt}{4ex}{\bf 3.4} & $\lambda'(1,1,1,1)=3$ \\
\hline
\end{tabular}\\
\vspace{1cm}
\caption{Conditions determining standard gap functions for $\delta\leq 3$}\label{table:vs}
\end{table}

\begin{proof}
Let $\mfm$ denote the ideal of $S$ generated by $t_1,\ldots,t_r$. 
By our assumption on $R'$, every element of $R'$ is of the form $c\cdot (1,\ldots,1)+s$ for some $c\in\KK$ and $s\in R'\cap \mfm$.
Let $R''$ be the span of all elements in $R$ arising (non-trivially) as products of elements of $R'\cap\mfm$. Then $R=R'+R''$. We let $\Sigma',\Sigma''$ denote the images of $R',R''$ under $\val$. The following observation is central to our argument:
\begin{rem}\label{rem:key2}
If $u\in \Sigma$, then either $u\in \Sigma'$, or $u\in \Sigma''$ or there exists $u',u''$ in $\Sigma',\Sigma''$ such that for all $i$, $u_i\geq \min (u_i',u_i'')$ with strict inequality  only if $u_i'=u_i''$.
\end{rem}

Our general strategy is to first show that if $\lambda_R=\lambda$, then $\lambda'$ must satisfy the conditions of Table~\ref{table:vs}.
We obtain conditions on $\lambda'$ in two ways. Firstly, $\lambda$ gives a lower bound for $\lambda'$. Secondly, using Remark~\ref{rem:sigma} we determine which $\alpha\in\ZZ_{\geq 0}^r$ are in $\Sigma$; using Remark~\ref{rem:key2}, we obtain some elements which \emph{must} be in $\Sigma'$, which then again by Remark~\ref{rem:sigma} gives conditions on $\lambda'$.

Once we have obtained necessary conditions on $\lambda'$, we observe conversely that $\lambda_R=\lambda$ holds. We again have two techniques. Firstly, the conditions on $\lambda'$ guarantee that $\Sigma'$ must or must not contain certain elements, which coupled with Remark~\ref{rem:key2} leads to conditions on $\lambda_R$. Secondly, using our classification of standard gap functions, we are able to rule out other possibilities.

The arguments are routine. We leave the details to the reader after illustrating them in several examples below. In the following, we will make use of the natural partial order on $\Sigma$. 
\end{proof}
\begin{ex}[{\bf 1.2}]
The set $\Sigma$ must contain $(1,1)$ so by Remark~\ref{rem:key2} $\Sigma'$ must contain $(1,1)$. The smallest element of $\Sigma''$ is $(2,2)$, and since $\Sigma$ contains $(2,1)$, $\Sigma'$ must contain $(k,1)$ for some $k\geq 2$.
It follows that $\lambda'(1,1)=\lambda'(2,2)=1$. Conversely, the only standard gap function in our classification satisfying this is {\bf 1.2}.
\end{ex}
\begin{ex}[{\bf 2.1.b} or \bf{3.1.d}]
In both cases, the set $\Sigma'$ must contain $2$ but not $1,3$ by Remark~\ref{rem:key2}. The condition on $\lambda'$ follows. Conversely, the only standard gap functions in our classification satisfying this condition are {\bf 2.1.b} and {\bf 3.1.d}.
It is impossible to differentiate between these two cases only using conditions on $\lambda'$. Indeed, one can take $R'$ to be the vector space generated by $1, t^2,t^4+ct^5,t^6,t^7$ for some $c\in \KK$. The function $\lambda'$ is independent of $c$, taking values $0,1,1,2,2,3,3,3,4,5,6,7,\ldots$. However, when $c=0$ we are in case {\bf 3.1.d} whereas when $c\neq 0$ we are in case {\bf 2.1.b}.
\end{ex}

\begin{ex}[{\bf 2.2.a} or {\bf 3.2.f}]
	In both cases, we see that $\Sigma'$ must contain $(1,1)$, but not $(1,k)$, $(k,1)$ for $k>1$ as these do not belong to $\Sigma$. This implies the required conditions on $\lambda'$ by upward propagation. It is straightforward to check that {\bf 2.2.a} and {\bf 3.2.f} are the only cases for which these conditions could be fulfilled.
It is impossible to differentiate between these two cases only using conditions on $\lambda'$. Indeed, one can take $R'$ to be the vector space generated by $1, t_1+t_2,t_1^3,t_1^2+t_2^2+c(t_1^2+t_2^3)$ for some $c\in \KK\setminus\{-1\}$. The function $\lambda'$ is easily seen to be independent of $c$.
However, when $c=0$ we are in case {\bf 3.2.f} whereas when $c\neq 0$ we are in case {\bf 2.2.a}.
\end{ex}

\begin{ex}[{\bf 2.2.b}]
	The smallest non-zero element of $\Sigma$ is $(2,1)$, so this must be in $\Sigma'$. This implies that the smallest element of $\Sigma''$ is $(4,2)$. Since $\Sigma$ contains $(3,1)$ and $(4,1)$, Remark \ref{rem:key2} implies that $\Sigma'$ contains $(3,1)$ and $(k,1)$ for some $k\geq 4$. Together, this implies that $\lambda'(2,1)=\lambda'(4,2)=2$.

	Conversely, if $\lambda'$ fulfills these conditions we are clearly in case {\bf 2.2.b}.
\end{ex}

\begin{ex}[{\bf 3.2.e}]
	The smallest non-zero element of $\Sigma$ is $(2,1)$, so this must be in $\Sigma'$. This implies that the smallest element of $\Sigma''$ is $(4,2)$, so $\Sigma'$ also contains $(3,1)$. Together, this implies that $\lambda'(2,1)=\lambda'(3,2)=\lambda'(4,1)=2$.
	Furthermore, we must have $\lambda'(4,2)=3$,  since $\lambda(4,2)=3$.

	Conversely, if $\lambda'$ fulfills these conditions we are clearly in case {\bf 3.2.e}.
\end{ex}

\begin{ex}[{\bf 3.3.b}]
	Since $\Sigma$ contains $(1,1,1)$, $\Sigma'$ must as well by Remark \ref{rem:key2}. The conditions on $\lambda'$ follow. Conversely, if $\lambda'$ fulfills these conditions, we clearly cannot be in case {\bf 3.3.a}; we cannot be in case {\bf 2.3} either as in the latter case $\lambda'(2,2,2)=2$. We rule out {\bf 3.3.c} as well since in that case, $\Sigma'$ must also contain e.g. $(1,1,k)$ for some $k\geq 2$ so $\lambda'(1,2,2)$ would be $2$.
\end{ex}

We conclude this section with an observation that will be useful for proving the relations of the singularities
   with families of Schubert varieties, as described in \S\ref{sec_schubert_conditions}. 

\begin{lemma}\label{lem:stronger_for_vs}
   Let $R'\subset S$ be a finite dimensional vector space containing a unit of $S$ 
      and leading to a gap function $\lambda'=\lambda_{R'}$. 
   If $\lambda'$ satisfies the conditions in one of the rows of Table~\ref{table:vs},
      then $\lambda'$ is a standard gap function.
\end{lemma}

\begin{proof}
   Since $R'$ contains a unit, it follows that $\lambda'(e_i)=0$ and $\lambda'(1, \dotsc, 1) \le r-1$. 
   If remains to show that $\lambda'(1, \dotsc, 1) \ge r-1$, 
      which follows since each row of Table~\ref{table:vs} contains a condition of the form 
      $\lambda'(\alpha)=|\alpha|-1$ for some $\alpha\in \NN^{r}$ 
      (in some cases this is satisfied directly for $\alpha = (1, \dotsc, 1)$).
\end{proof}

\section{Singularities from Projections}\label{sec:project}
\subsection{Setup and First Results}
We now finally return to the geometric situation of considering the singularities of a curve that arise via projection.
Fix a smooth projective curve $X$ with a very ample line bundle $\mcL$ of degree $d$. We will use notation as in \S \ref{sec:osc}.
In particular, set $W=H^0(X,\mcL)$ with dual space $V$.
Since we are assuming that $\mcL$ is very ample, we can view $X$ as being embedded in $\PP(V)$.

We are interested in stratifying the Grassmannian of codimension-$(n+1)$-planes $L$ in $V$ according to the singularities of $\phi(X)$, where $\phi:X\to \PP^n$ is the projection with center $\PP(L)$.
As in the statement of Theorem \ref{thm:2}, $\ell=\dim V-(n+1)$, that is, it will be the dimension of $L$.

\begin{lemma}\label{lemma:bir}
Let $L\subset V$ be a codimension-$(n+1)$ subspace, and $M=L^\perp\subset W$.
Then $M$ is basepoint free if and only if $\PP(L)\cap X=\emptyset$. 
Furthermore, assuming that $M$ is basepoint free:
\begin{enumerate}
	\item $\phi$ is generically one-to-one if and only if $\PP(L)$ intersects only finitely many secant lines of $X$.\label{cl:1}
	\item $\phi$ is generically unramified if and only if $\PP(L)$ intersects only finitely many tangent lines of $X$.\label{cl:2}
	\item $\phi$ is birational onto its image if and only if $\PP(L)$ intersects only finitely many secant and tangent lines  of $X$.\label{cl:3}
\end{enumerate}
\end{lemma}
\begin{proof}
A point $P\in X$ is a basepoint of $M$ if and only if every section of $M$ vanishes at $P$, or equivalently, the point $P$ is contained in $\PP(L)$; the claim regarding basepoint freeness of $M$ follows.

Assuming that $\PP(L)\cap X=\emptyset$, we note that points $P_1,P_2\in X$ are identified under the projection if and only if $\PP(L)$ meets the secant line through $P_1,P_2$. Claim~\ref{cl:1}
follows.
Likewise, the map $\phi$ has vanishing differential at $P\in X$ if and only if $\PP(L)$ meets the tangent line though $P$; claim~\ref{cl:2} follows. Claim~\ref{cl:3} now follows since being birational onto the image is the same as being generically one-to-one and unramified. 
\end{proof}

\begin{lemma}\label{lemma:sec}
Assume that $\PP(L)\cap X=\emptyset$ and $2\ell < d-2\rho_g$. Then $\PP(L)$ only intersects finitely many secant and tangent lines of $X$.
\end{lemma}
\begin{proof}
	We first show that $\PP(L)$ intersects at most $\ell$ tangent lines. Indeed, suppose instead that $\PP(L)$ intersects tangent lines through  $P_1,\ldots,P_{\ell+1}\in X$.
The span of the preimages of these tangent lines in $V$ is $\langle V^2(P_1), \ldots, V^2(P_{\ell+1})\rangle$. Since $2(\ell+1)<d-2\rho_g+2$, 
taking $\F^\alpha=\langle V^2(P_1), \ldots, V^2(P_{\ell+1})\rangle$ in Lemma~\ref{lemma:bound} yields that these tangent lines are independent. For $L$ to intersect them all non-trivially, we must have $\dim L\geq \ell+1$, a contradiction.

	If $\PP(L)$ intersects $\ell+1$ secant lines through disjoint pairs of points, we would arrive at a similar contradiction; 
	here we are applying Lemma~\ref{lemma:bound} to
	\[\F^\alpha=\langle V^1(P_1),V^1(P_1'), \ldots, V^1(P_{\ell+1}),V^1(P_{\ell+1}')\rangle\] for distinct points $P_1,P_1',\ldots,P_{\ell+1},P_{\ell+1}'$.
	So it remains to show that $\PP(L)$ cannot intersect infinitely many secant lines through a single point $P$. To that end, suppose that $\PP(L)$ intersects the $\ell+1$ secant lines through $P$ and $P_i$, $i=1,\ldots,\ell+1$. Applying Lemma~\ref{lemma:bound} with
	\[\F^\alpha=\langle V^1(P),V^1(P_1), \ldots, V^1(P_{\ell+1})\rangle,\] 
 the span in $V$ of the preimages of these secant lines has dimension $\ell+2$. Since we have assumed that $\PP(L)$ does not intersect $P$, we must again have $\dim L \geq \ell+1$, a contradiction.
\end{proof}

By Lemmas~\ref{lemma:bir} and~\ref{lemma:sec}, we obtain statement (\ref{item:bir}) of Theorem~\ref{thm:2}. Based on this, we will henceforth assume that $\PP(L)\cap X=\emptyset$ and $2\ell<d-2\rho_g$. 
Let
\[
\{P_{ij}\}_{\substack{1\leq i \leq m\\ 1\leq j\leq r_i}}
\]
be the finite set of points of $X$ such that $L$ intersects a tangent or a secant line through each $P_{ij}$. We have indexed the $P_{ij}$ such that the secant line between $P_{ij}$ and $P_{i'j'}$ intersects $L$ if and only if $i=i',j\neq j'$.
Let $C\subset \PP^n$ be the image of $X$ obtained by projecting from $L$. Then the  
singularities $Q_i$ of $C$ are indexed by $i=1,\ldots,m$.
We will refer to the points $P_{ij}$ as the \emph{ramification points} of $X$.

The following theorem implies statement  (\ref{item:mult}) of Theorem~\ref{thm:2}.

\begin{thm}\label{mainthmproject}
Assume that $\ell=\dim L\leq 3$ and $2\ell < d-2\rho_g$. 
Then the singularity type of the point $Q_i\in C$ is determined from the conditions of Proposition~\ref{prop:linear} applied to the function	
\begin{align*}
		\lambda':\ZZ_{\geq0}^{r_i}&\to \ZZ_{\geq 0}\\
					\alpha&\mapsto\dim L\cap \F^\alpha(P_{i1},P_{i2},\ldots,P_{ir_i}), 
\end{align*}
and the classification in Proposition~\ref{prop:sing}.
\end{thm}
\begin{proof}

Theorem~\ref{thm:b1} guarantees that $\delta = \rho_a-\rho_g$ is bounded by $\ell\leq 3$.
The same bound holds for the values of the gap function $\lambda_i$ for $R_i$, the completion of the local ring at $Q_i$.

By Corollary~\ref{cor:l1}, the function $\lambda'$ defined in the statement agrees with the gap function $\lambda_{(1/s)M}$ of $(1/s)M$ in $S_i$ for all $\alpha$ with $|\alpha|\leq d+1-2\rho_g$.  Note that the assumptions imply $d+1-2\rho_g\geq 2\ell +2\geq 2\delta +2$.

The gap function of $(1/s)M$ in $S_i$ is a standard gap function: indeed, since $s\in M$,\
   we must have $\lambda_{(1/s)M}(1,\dotsc,1) \le r_i-1$ and $\lambda_{(1/s)M}(e_j)=0$ for all $j$. 
But also any element $f/s$ of $(1/s)M$ whose valuation is not $(0,....0)$ 
  must have valuation $(\alpha_1....\alpha_r)$ where no $\alpha_k=0$. 
Otherwise the section $f$ would vanish at some $P_{ij}$ but not at others, contradicting the assumption that the $P_{ij}$ all map to the same $Q_i$.
Thus also  $\lambda_{(1/s)M}(1,\dotsc,1) \ge r_i-1$.

In Table~\ref{table:vs}, the conditions for $\lambda'(\alpha)$ are within the range $|\alpha|\leq 2\delta + 2$. Moreover, the completion of the algebra generated by $(1/s)M$ in the ring $S_i$ coincides with $R_i$, since the curve $C$ is parametrized by the functions in $M$. Hence, applying Proposition~\ref{prop:linear} gives the statement. The possible singularity types are classified in Proposition~\ref{prop:sing}. 
\end{proof}

\begin{rem}[Caveat and higher dimensions]\label{rem:caveat}
	We note a caveat to Theorem~\ref{mainthmproject}: the result does not let us differentiate between the classes {\bf 2.1.b} (a rhamphoid cusp) and {\bf 3.1.d} (a $(2,7)$ cusp).
	Likewise, we cannot differentiate between {\bf 2.2.a} and {\bf 3.2.f}.

	One might wonder what can be said when $\dim L> 3$. In principle, the same classification presented in \S\ref{sec:classify} can be carried out for $\delta >3$, however it will be much more complicated. Furthermore, the situation in the caveat mentioned above will occur much more frequently: 
	 we will not be able to differentiate between many different classes of singularities by looking only at gap functions corresponding to vector spaces (as opposed to those of local rings).
	Finally, for more complicated singularities, classification via gap functions of their local rings will lead to non-discrete classes of singularities.
	For example, a planar quadruple point whose four branches have  pairwise linearly independent tangent directions occurs in a one-dimensional family.
\end{rem}
\subsection{Schubert Conditions}\label{sec_schubert_conditions}
As in Theorem~\ref{mainthmproject}, let $1\leq \ell\leq 3$ with $2\ell<d-2\rho_g$; we also assume $n>2$. Additionally, let $U\subset G(\ell,V)$ 
be the locus of $\ell$-planes such that $\PP(L)$ does not meet  $X$.
For each $i=1,\ldots,m$, we fix a singularity type with singularity degree $\delta_i$.
By Theorem~\ref{thm:b1}, we may assume that  $\sum_{i=1}^m \delta_i\leq \ell$, and hence $m\leq 3$.
In the rest of the section, we analyze the locus of those linear spaces $L$ in $U$ whose corresponding projection gives rise to such singularities. 

To achieve this, for each $i=1,\ldots, m$, we first fix distinct ramification points $P_{i1},\ldots,P_{i{r_i}}$ of $X$; these will be the points forming the fiber
at $Q_i$ of the projection.

For fixed $i$, we consider the conditions in Table~\ref{table:vs} applied to the function
\begin{align}\label{equ:define_lambda'_from_multifiltration}
\lambda':\alpha\mapsto \dim L\cap \F^\alpha(P_{i1},P_{i2},\ldots,P_{ir_i}).
\end{align}
These give a locally closed condition in $G(\ell,V)$. Removing the open conditions, one may check case by case that the closure of this locus is a Schubert variety 
\[\mfS_i=\mfS_i(P_{i1},\ldots,P_{ir_i})\] for a flag obtained from subspaces of the form $\F^\alpha(P_{i1},P_{i2},\ldots,P_{ir_i})$.
By inspecting Table~\ref{table:vs},  one finds that only linear spaces of  dimension at most $2\delta_i$ are required to be chosen when defining $\mfS_i$; the rest
can be chosen arbitrarily. 

As in the proof of Theorem~\ref{mainthmproject}, 
   for any $L\in U \subset G(\ell,V)$ 
   the function $\lambda'$ defined in \eqref{equ:define_lambda'_from_multifiltration} 
   agrees with the gap function $\lambda_{(1/s) M}$ for small values of $\alpha$.
If $\lambda'$ satisifes any  of the conditions from Table~\ref{table:vs},
   we must have that $\lambda_{(1/s) M}$ is a standard gap function by Lemma~\ref{lem:stronger_for_vs}.
In particular, $\phi$ maps all points $P_{i1},P_{i2},\ldots,P_{ir_i}$ to a single point $Q_i \in C$.
   Assuming that no other points $P\in X$ map to $Q_i$, Proposition~\ref{prop:linear} implies that the singularity of $C$ at $Q_i$ has the corresponding type as listed in Table~\ref{table:vs}.

The partitions corresponding to the Schubert variety $\mfS_i$ 
   are found to be those featured in Tables~\ref{table:strat1} and~\ref{table:strat2};
The codimension of $\mfS_i$ in $G(\ell,V)$ is the size of this partition. 
Furthermore, it is straightforward to check that the closed locus being removed from $\mfS_i$ by the open conditions is a proper subset, being contained in the union of a finite number of Schubert varieties $\mfS_i'$, all properly contained in $\mfS_i$.
Again, inspecting Table ~\ref{table:vs}, one finds that in order to define these loci $\mfS_i'$, we only need to specify linear spaces in the flag up to dimension $2\delta_i+2$. 
Denote by $\mfS_i^\circ$ the open set of $\mfS_i$ where the conditions on $\lambda'$ are fulfilled.

We illustrate this with two examples:
\begin{ex}[{\bf 2.3}]
The closed condition in this case is that 
\[
\dim L\cap \F^{(1,1,1)}(P_{i1},P_{i2},P_{i3})\geq 2.
\]
Consider any flag of $V$ beginning with  
\[
0\subset \F^{(1,0,0)}(P_{i1},P_{i2},P_{i3})
\subset \F^{(1,1,0)}(P_{i1},P_{i2},P_{i3})
\subset \F^{(1,1,1)}(P_{i1},P_{i2},P_{i3})
\subset\ldots.
\]
The locus of $G(\ell,V)$ given by the closed condition above is exactly the Schubert variety $\mfS$ corresponding to the partition
\[
\textrm{\tbc}
\]
with respect to this flag.

The locus that we are removing by the open condition is the locus where
\[
\dim L\cap \F^{(2,2,2)}(P_{i1},P_{i2},P_{i3})\geq 3.
\]
However, this is the Schubert variety $\mfS'$ corresponding to the partition
\[
\textrm{\tbcf}.
\]
and so it is a proper subset.
\end{ex}
\begin{ex}[{\bf 3.2.e}]
The closed condition in this case is that 
\begin{align*}
\dim L\cap \F^{(1,1)}(P_{i1},P_{i2})\geq 1;\\
\dim L\cap \F^{(2,1)}(P_{i1},P_{i2})\geq 2;\\
\dim L\cap \F^{(4,2)}(P_{i1},P_{i2})\geq 3.
\end{align*}
Consider any flag of $V$ beginning with  
\begin{align*}
0&\subset \F^{(1,0)}(P_{i1},P_{i2})
\subset \F^{(1,1)}(P_{i1},P_{i2})
\subset \F^{(2,1)}(P_{i1},P_{i2})\\
&\subset \F^{(2,2)}(P_{i1},P_{i2})
\subset \F^{(3,2)}(P_{i1},P_{i2})
\subset \F^{(4,2)}(P_{i1},P_{i2})
\subset\ldots.
\end{align*}
The locus of $G(\ell,V)$ given by the closed condition above is exactly the Schubert variety $\mfS$ corresponding to the partition
\[
\textrm{\tbcf}
\]
with respect to this flag.

The locus that we are removing by the open condition is the locus where
\[
\dim L\cap \F^{(3,2)}(P_{i1},P_{i2})\geq 3
\]
or
\[
\dim L\cap \F^{(4,1)}(P_{i1},P_{i2})\geq 3.
\]
Each of these is a Schubert variety $\mfS'$ corresponding to the partition
\[
\textrm{\tbce}.
\]
and so it is a proper subset. (In the second case we must modify the flag so it includes $\F^{(4,1)}$ instead of $\F^{(3,2)}$).
\end{ex}

For each $i$, let $\Delta_i \subset G(\ell, V)$ be the locus where for any point $P$ distinct from $P_{i1},\ldots,P_{ir_i}$, and for some $j$ we have $\dim L\cap \F^{(1,1)}(P, P_{ij})\geq 1$.
In other words, $\Delta_i$ is the set of those $L$ that intersect a secant line to $X$ passing through exactly one of the points $P_{i1},\ldots,P_{ir_i}$.
By Theorem~\ref{mainthmproject}, $(\mfS_i^\circ\setminus\Delta_i)\subset U$ is exactly the locus of those linear spaces $L\in U$ such that the $i$th singularity type arises with fiber $P_{i1},\ldots,P_{ir_i}$ under the projection corresponding to $L$.

\begin{lemma}\label{Schubertloci}
Assume $n>2$, $1\leq \ell\le 3$, and $2\ell < d -2\rho_g$. Then 
\[
\mfS = \bigcap_{i=1}^m   (\mfS_i^\circ\setminus\Delta_i)
\]
is non-empty, and of codimension equal to the sum of codimensions of the varieties $\mfS_i$ in $G(\ell,V)$.
\end{lemma}
\begin{proof}
Since $\sum \delta_i\le 3$, by Theorem~\ref{thm:b1} we must have $\delta_i=1$ for all indices $i$ except at most one. As noted above, the flag we need to define each $\mfS_i$ is only determined up to dimension $2\delta_i$; its higher dimension subspaces may
be chosen arbitrarily. 

Consider the intersection $\bigcap_i \mfS_i$. Since all of the varieties $\mfS_i$, except possibly one, correspond to {\bf 1.1} or {\bf 1.2}, we may calculate their intersection product using the {\it Pieri rule} (possibly repeatedly). Inspecting Tables~\ref{table:strat1} and~\ref{table:strat2} shows that the intersection product is non-zero, hence these Schubert varieties have non-trivial intersection.

Since $\sum \delta_i\leq 2\ell \leq d+1-2\rho_g$, the second part of Lemma~\ref{lemma:bound} states that the $m$ defining flags may be taken to lie in relative general position. Hence, the intersection of the $\mfS_i$ has the expected dimension, i.e., its codimension is the sum of the codimensions.
This implies the statement once we establish that the intersection of the loci $(\mfS_i^\circ\setminus\Delta_i)$ is in fact non-empty.

To see this, we will momentarily focus on $\mfS_i^\circ$ for $i=1$, and then permute indices.
The complement of $(\mfS_1^\circ\setminus\Delta_1)$ in $\mfS_1$ is contained in the union of the following varieties: 

\begin{enumerate}[(i)]
\item\label{i:1}
 Schubert varieties obtained by changing the open conditions of Table~\ref{table:vs} to closed ones (for $i=1$);
 \item\label{i:2} Loci obtained by requiring the conditions for $\mfS_1$ as well as \[\dim L\cap \F^{(1,1)}(P,P_{1j})\geq 1\] for some $j=1,\ldots,r_1$ and  $P\in X$ distinct from 
	$P_{11},\ldots,P_{1r_1}$;
\end{enumerate}

The loci in (\ref{i:1}) are the boundaries of the open conditions described in Table~\ref{table:vs}; those in (\ref{i:2}) are $\mfS_1\cap \Delta_1$. 
We now analyze these loci more closely to show the statement. 

Type (\ref{i:1}). By inspection of Table~\ref{table:vs} we derive that the Schubert varieties $\mfS'$ arising from this condition are determined by specifying a flag only up to dimension $2\delta_1+2$. Since $\sum 2\delta_i+2\leq d+1-2\rho_g$, we may again use the second part of Lemma~\ref{lemma:bound} to be able to choose a flag for $\mfS'$ and flags for $\mfS_j$, $j\neq 1$, that are in relative general position. Therefore the locus
$
\mfS' \cap \bigcap_{j=2}^m \mfS_j
$
has the expected dimension; this is strictly smaller than the dimension of $\bigcap_{i=1}^m \mfS_i$. 

Type (\ref{i:2}). Suppose that $\mfS_1$ is specified by a flag where the largest dimensional linear space imposing a condition on $L$ is $\F^\alpha(P_{11},\ldots,P_{1r_1})$; we have $|\alpha|\leq 2\delta_1$. For fixed $P$ and $j$, this second type of locus is contained in a Schubert variety $\mfS'$ where we have specified all the conditions for $\mfS_1$, along with the condition
\[
\dim L\cap \F^{(\alpha,1)}(P_{11},\ldots,P_{1r_1}, P)>
\dim L\cap \F^{\alpha}(P_{11},\ldots,P_{1r_1}).
\]

If $\delta_1 = 3$, this is not possible, as $\dim L\cap \F^{\alpha}(P_{11},\ldots,P_{1r_1})$ is already $3$. Otherwise, the corresponding partition is the same as for $\mfS_1$, except that the bottom row is repeated. These Schubert varieties $\mfS'\subset \mfS_1$ are determined by specifying a flag only up to dimension $2\delta_1+1$. Again, by Lemma~\ref{lemma:bound} the flags for $\mfS'$ along with the flags for the Schubert varieties $\mfS_j$ for $j\neq 1$ can be chosen to be in relative general position. Hence the intersection of these Schubert varieties has the expected dimension; its codimension in $\bigcap_{i=1}^m \mfS_i$ is at least $n-1$. Letting $P\in X$ vary, we obtain a locus of codimension at least $n-2>0$, again a proper closed subset.

In conclusion, removing the loci (\ref{i:1}) and (\ref{i:2}) from $\mfS_1$ produces an open subset 
   $U'= \mfS_1^\circ\setminus\Delta_1 \subset \mfS_1$. 
   Above, we showed that the intersection of $U'$ with the varieties $\mfS_i$ for $i> 1$
must be non-empty.
Permuting the indices and taking the intersection, we see that
\[
\mfS = \bigcap_{i=1}^m (\mfS_i^\circ\setminus\Delta_i)
\]
must also be non-empty. This concludes the proof. 
 \end{proof}

Lemma~\ref{Schubertloci} shows that, for any configuration of singularities with $\sum_{i=1}^m \delta_i\leq \ell$, the locus of $U$ for which the induced projection has exactly these singularities with $P_{i1},\ldots,P_{ir_i}$ as ramification points is non-empty and of the expected dimension.
In particular, statement (\ref{item:sing}) of Theorem~\ref{thm:2} follows. However, in order to obtain the locally closed subvariety of $U$ leading to a fixed configuration of singularities, we have to allow the ramification points $P_{i1},\ldots,P_{ir_i}$ to vary. This gives us $r=\sum r_i$ parameters, leading to a family $\mathcal{Y}\subset U$ whose dimension is
\[
r+\dim \bigcap_{i=1}^m \mfS_i.
\]
Indeed, the family comes with a dominant morphism $\mathcal{Y}\to X^{(r)}$
obtained by mapping $L$ to the corresponding ramification points, where $X^{(r)}$ is the $r$-th symmetric product of the curve $X$. The fibers of this map are just unions of the 
\[
\bigcap_{i=1}^m (\mfS_i^\circ\setminus\Delta_i),
\]
obtained by permuting the points in the fiber at $Q_i$ of the projection from $L$.
Since the $\mfS_i$ intersect dimensionally transversally, we obtain that 
\[\dim \mathcal{Y}=\dim G(\ell,V)-\sum_{i=1}^m(\codim \mfS_i-r_i).\]
Statements (\ref{item:flags}) and (\ref{item:codim}) of Theorem~\ref{thm:2} then follow from the observation that the codimension of $\mfS_i$ is just the number of boxes in the corresponding partition.
The proof of Theorem~\ref{thm:2} is now complete.

\bibliographystyle{amsplain}
\bibliography{paper}
\end{document}